\documentclass[final,leqno]{siamltex704}
\usepackage{amsmath}
\usepackage{amssymb}
\usepackage{graphicx}
\usepackage{subfig}
\usepackage{bm}
\usepackage{booktabs}
\usepackage{multirow}
\usepackage[notcite,notref]{showkeys}
\usepackage[numbers,sort&compress]{natbib}
\usepackage[figuresright]{rotating}
\usepackage{verbatim}
\usepackage{color,xcolor}

\def\T{{\mathcal T}}
\def\E{{\mathcal E}}
\def\3bar{{|\hspace{-.02in}|\hspace{-.02in}|}}
\def\ljump{{[\hspace{-0.02in}[}}
\def\rjump{{]\hspace{-0.02in}]}}

\newtheorem{remark}{Remark}[section]
\newtheorem{example}{\bf Example}
\newtheorem{algorithm}{Algorithm}

\title{A weak Galerkin method and its two-grid algorithm for the quasi-linear elliptic problems of non-monotone type}

\author{Peng Zhu\thanks{College of
Data Science, Jiaxing University, Jiaxing, Zhejiang 314001, China
(pzh@zjxu.edu.cn). This research was supported in part by Zhejiang Provincial Natural Science Foundation of China (LY19A010008) and Natural Science Foundation of China (12071184).}
\and
Shenglan Xie\thanks{College of Information Engineering, Jiaxing Nanhu University, Jiaxing, Zhejiang 314001, China (slxie@jxnhu.edu.cn).}
}

\begin{document}

\maketitle

\begin{abstract}
In this article, a  weak Galerkin (WG)  method is firstly presented and analyzed for the quasi-linear elliptic problem of non-monotone type. By using Brouwer's fixed point technique, the existence of WG solution and error estimates in both the energy norm and the $L^2$ norm are derived. Then an efficient two-grid WG method is introduced to improve the computational efficiency. The convergence error of the two-grid WG method is analyzed in the energy norm. Numerical experiments are presented to verify our theoretical findings.
\end{abstract}


\pagestyle{myheadings}

\section{Introduction}

Consider the quasi-linear elliptic problem of the form
\begin{subequations}
\begin{eqnarray}
-\nabla\cdot(a(x,u){\nabla u})&=&f,\quad \mbox{in}\;\Omega,\label{pde}\\
u&=&g,\quad \mbox{on}\;\partial\Omega,\label{pde-bc}
\end{eqnarray}
\end{subequations}
where $\Omega$ is a bounded polygonal domain in $\mathbb{R}^2$. We suppose that $a(x,u)$ is a twice continuously differentiable function in $\overline{\Omega}\times \mathbb{R}$ and all the derivatives of $a(\cdot,\cdot)$ through second order are bounded in $\overline{\Omega}\times \mathbb{R}$. We also assume that there exist positive constants $\alpha_0, \alpha_1$ such that
\[0<\alpha_0\leq a(x,u)\leq \alpha_1.\]
For sake of notational simplicity, we denote $a(x,u)$ by $a(u)$ in the rest of this paper.

The weak Galerkin (WG) method is a novel finite element framework for solving partial differential equations. It was introduced in \cite{wy,wy-mixed} for the numerical solution of second order elliptic problems. Since then, there has been considerable interest in WG methods for the numerical solution of a wide range of partial differential equations. We refer the reader to \cite{shw18,gm,ggz,lyzz,wg-poly,wg-red,wg-bi1,wg-bi2,wg-bi3,mwy-helm,mwyz-maxwell,mwyz-interface,wg-bi4,yzz,yz,sf-wg,wg-bi6,yzz1}, and the references therein for details.  However, there are few papers that are concerned to the nonlinear elliptic problems.  To the best of our knowledge, only references \cite{wg-poly,shw18,yzz1} conducted their investigations in this direction.

For the second order elliptic quasi-linear PDEs, the existence of solutions of the WG methods was shown in \cite{wg-poly} by a Schauder's fixed point argument. However, the uniqueness and the error estimations of the numerical approximations are restricted only to the linear PDEs, and have not been addressed for the nonlinear ones. Recently in \cite{shw18}, the authors gave the well-posedness and error estimate in the energy norm for the monotone quasi-linear PDEs. Most recently, \cite{yzz1} investigated the stabilizer-free WG methods \cite{yz,sf-wg} on polytopal meshes for a class of quasi-linear elliptic problems of monotone type.

Direct application of the techniques presented in \cite{shw18,yzz1} to establish the a priori error estimates of the WG solutions for non-monotone type quasi-linear elliptic problem is not possible because the associated differential operators may not satisfy monotonicity property. Non-monotone type quasi-linear elliptic problems have been studied by conforming finite element methods \cite{dd}, mixed finite element methods \cite{m85,park95}, discontinuous Galerkin methods \cite{gp, gnp}.

The first aim of this paper is to make an effort for conducting a priori error analysis of the WG method for solving non-monotone quasi-linear elliptic problems. Following the ideas in \cite{os07}, we employ Brouwer's fixed point theorem to prove the existence of the discrete solution, and in turn derived the a priori error estimates in a mesh-dependent energy norm and in the $L^2$ norm.

As it is well known, the two-grid method is an efficient algorithm for solving nonlinear partial differential equations \cite{xu94,xu96}.  Two-grid finite volume element method \cite{bg07} and two-grid discontinuous Galerkin method \cite{bg11} were further investigated for non-monotone quasi-linear elliptic problems. Another aim of this paper is to propose an efficient WG method by adopting the two-grid idea for the quasi-linear elliptic problem. Convergence of the two-grid WG method is rigorously analyzed.
As far as we know, the present work is the first attempt to apply and analyze the two-grid technique in the setting of the WG method.

The outline of this paper is as follows. Section 2 introduces the WG method for the problem (\ref{pde})-(\ref{pde-bc}). In Section 3, we derive optimal order error estimates of the WG method in both the energy norm and the $L^2$ norm. The two-grid algorithm of the WG method is proposed in Section 4, which is then followed by derivation of its convergence analysis in the energy norm. Section 5 carries out numerical experiments to verify our theoretical findings. A summary is given in Section 6.

\section{Weak Galerkin Finite Element Methods}\label{Sec:wg-fem}

First of all, let us introduce the concept of weak function.
Let $K$ be any polygonal domain with boundary $\partial K$. A weak function on $K$ refers to a function $v=\{v_0, v_b\}$ such that $v_0\in L^2(K)$ and $v_b \in H^{1/2}(\partial K)$, where $v_0$ means the value of $v$ in $K$, and $v_b$ represents the boundary value of $v$. Note that $v_b$ may not necessarily be related to the trace of $v_0$ on $\partial K$. Let $W(K)$ be the set of weak functions on $K$, i.e.,
\[ W(K):=\{v=\{v_0,v_b\}: v_0\in L^2(K), v_b\in H^{1/2}(\partial K)\}.\]
It is worth to point out that a function $v\in H^1(K)$ can be viewed as a weak function $\{v_0, v_b\}$ of $W(K)$ with $v_0=v|_K$ and $v_b=v|_{\partial K}$.

Let ${\mathcal T}_h$ be a shape regular polygonal partition of the domain $\Omega$ \cite{wy-mixed}.
Denote by ${\cal E}_h$ the set of all edges in ${\cal
T}_h$, and let ${\cal E}_h^0={\cal E}_h\backslash\partial\Omega$ be
the set of all interior edges. For each element $K\in\T_h$, we denote by $h_K$ the diameter of $K$, i.e., $h_K=\mathrm{diam}(K)$. Similarly, let $h_e$ be the length of an edge $e\in\E_h$.

On this partition of $\T_h$, we introduce the following broken Sobolev space
\begin{align*}
H^m(\T_h):=\{v\in L^2(\Omega): v|_K\in H^1(K), \;\forall K\in\T_h\},
\end{align*}
for any integer $m\geq 0$.

Let $\mathbb{P}_m(D)$ denote the set of polynomials defined on $D$ with degree no more than $m$, where $D$ may be an element $K$ of $\T_h$ or an edge $e$ of $\E_h$. In what follows, we often consider the broken polynomial spaces
\[\mathbb{P}_m(\T_h):=\{v\in L^2(\Omega): v|_K\in \mathbb{P}_m(K), \,\,\forall K\in\T_h\},\]
and
\[\mathbb{P}_m(\E_h):=\{v\in L^2(\E_h): v|_{e}\in \mathbb{P}_m(e), \,\,\forall e\in\E_h\}.\]

Then a WG finite element space $V_h$ associated with $\T_h$ for $k\geq 1$ is defined by
\begin{equation}
V_h=\{v=\{v_0, v_b\}:\ v_0\in \mathbb{P}_{k}(\T_h),
 v_b\in \mathbb{P}_{k}(\E_h)\}.
\end{equation}

Denote by $V_h^0$ a subspace of $V_h$ with vanishing traces,
\begin{align}\label{Vh0}
V_h^0=\{v=\{v_0,v_b\}\in V_h, \ v_b|_e=0,\
e\subset\partial K\cap\partial\Omega\}.
\end{align}

For the sake of simplicity, we introduce the following notations: for any $v, w\in H^1(\T_h)$,
\begin{align*}
(v,w)_{\T_h} =\sum_{K\in\T_h}(v,w)_K,\quad \mbox{where } (v,w)_K=\int_K vw \mathrm{d}\bm{x},
\end{align*}
and
\begin{align*}
 \langle v,w\rangle_{\partial\T_h}&=\sum_{K\in\T_h} \langle v,w\rangle_{\partial K},
 \quad \mbox{where } (v,w)_{\partial K} =\int_{\partial K} vw \mathrm{d}s.
\end{align*}

\begin{definition}[Weak Gradient] For any function $v=\{v_0, v_b\}\in V_h$, its weak gradient $\nabla_{w}v$,
is piecewisely defined as the unique polynomial $(\nabla_{w}v)|_K \in [\mathbb{P}_{k-1}(K)]^2$ such that
\begin{equation}\label{wl}
(\nabla_{w} v, \bm{\phi})_K = -(v_0, \nabla\cdot\bm{\phi})_K+\langle v_b, \bm{\phi}\cdot\bm{n}\rangle_{\partial K},\quad
\forall \bm{\phi}\in [\mathbb{P}_{k-1}(K)]^2,
\end{equation}
for any $K\in\mathcal{T}_h$.
\end{definition}

Now, we are ready to present our weak Galerkin finite element method for the problem (\ref{pde})-(\ref{pde-bc}).

\begin{algorithm}[The WG method]\label{alg:nonlin}
The  weak Galerkin finite element scheme for solving problem (\ref{pde})-(\ref{pde-bc}) is defined as follows: find $u_h=\{u_0,u_b\}\in V_h$ such that $u_b=\mathcal{Q}_bg$ on $\partial\Omega$ and the following equation
\begin{equation}\label{wg}
A_h(u_h; u_h, v_h)=(f,\;v_0), \quad\forall\
v_h=\{v_0, v_b\}\in V_h^0,
\end{equation}
where
\begin{align}\label{bilin-form}
A_h(u_h; v_h, w_h) = (a(u_0)\nabla_{w} v_h, \nabla_{w} w_h)_{\T_h}+s_h(v_h, w_h),
\end{align}
with the stabilization term $s_h(\cdot,\cdot)$ is defined by
\[
s_h(v_h, w_h):=\sum_{K\in\T_h}h_K^{-1}\langle v_0-v_b, w_0-w_b\rangle_{\partial K}, \quad \forall v_h, w_h\in V_h,
\]
where $\mathcal{Q}_b$ is an $L^2$ projection defined by (\ref{Qb}).
\end{algorithm}

\section{Error analysis}\label{Sec:err-eqn}
In this section, we shall derive the a priori error estimates of the WG method (\ref{wg}) for solving the quasi-linear elliptic problem (\ref{pde})-(\ref{pde-bc}). To this end, we firstly introduce notations and some useful lemmas in Sect. 3.1 and then derive an error equation for the WG method (\ref{wg}) in Sect. 3.2. Using this error equation, a fixed point mapping is constructed and discussed in Sect. 3.3. Finally, by the Brouwer's fixed point theorem the error estimates in both energy norm and $L^2$ norm are concluded in Sect. 3.4.

\subsection{Preliminary}

In order to analyze the WG method (\ref{wg}), we introduce the energy norm $\3bar\cdot\3bar$ over $V_h^0$ by
\begin{equation}\label{3bar-norm}
\3bar v\3bar=(\sum_{K\in\T_h}\3bar v\3bar_K^2)^{1/2},
\end{equation}
with
\begin{align*}
\3bar v\3bar_K=(\|\nabla_wv\|_{L^2(K)}^2+h_K^{-1}\|v_0-v_b\|_{L^2(\partial K)}^2)^{1/2},
\end{align*}
and we also need the $H^1$-like norm $\|\cdot\|_{1,h}$
\begin{equation}\label{H1-norm}
\|v_h\|_{1,h} =(\sum_{K\in\T_h}\|\nabla v_h\|_{1,h,K}^2)^{1/2},
\end{equation}
for all ${v_h}\in V_h + H^1(\T_h)$, where
\begin{align*}
\|\nabla v_h\|_{1,h,K}=
(\|\nabla v_0\|_{L^2(\T_h)}^2+h_K^{-1}\|v_0-v_b\|_{L^2(\partial K)}^2)^{1/2}.
\end{align*}
It is easy to see that $\|\cdot\|_{1,h}$ is indeed a norm on the finite element space $V_h^0$.

\begin{lemma}\label{lem:happy} There exist two positive constants $C_1$ and $C_2$ such
that for any $v=\{v_0,v_b\}\in V_h$, we have
\begin{equation*}
C_1 \|v\|_{1,h}\le \3bar v\3bar \leq C_2 \|v\|_{1,h}.
\end{equation*}
\end{lemma}
\begin{proof}
The proof is trivial. The interesting reader is referred to the proof of Lemma 5.3 in \cite{wg-red}.
\end{proof}

In what follows, the trace inequality and inverse inequality are frequently used in our analysis, which state as \cite{ern}: for any $p, q\in [1, \infty)$, there holds
\begin{equation}\label{trace1}
\|\phi\|_{L^p(\partial K)} \leq C \left( h_K^{-1/p} \|\phi\|_{L^p(K)} + h_K^{1-1/p}
\|\nabla \phi\|_{L^p(K)}\right), \quad \forall \phi\in W_p^1(K),
\end{equation}
and
\begin{equation}\label{trace2}
\|\phi_h\|_{L^p(\partial K)} \leq Ch_K^{-1/p} \|\phi_h\|_{L^p(K)}, \quad \forall \phi_h\in \mathbb{P}_k(K),
\end{equation}
and
\begin{equation}\label{inv}
\|\phi_h\|_{L^p(K)} \leq Ch_K^{2(1/p-1/q)} \|\phi_h\|_{L^q(K)}, \quad \forall \phi_h\in \mathbb{P}_k(K).
\end{equation}
{Especially, for $p=\infty$ and $q=2$, we have
\begin{equation}\label{inv2}
\|\phi_h\|_{L^{\infty}(K)} \leq Ch_K^{-1} \|\phi_h\|_{L^2(K)}, \quad \forall \phi_h\in \mathbb{P}_k(K).
\end{equation}
}

In \cite{b03, ls03}, the following Poincar\'{e} type inequality  has been proved  for discontinuous functions in the broken Sobolev space $H^1(\T_h)$.
\begin{lemma}[Poincar\'{e} type inequality]\label{lem:P-dg}
For any $v\in H^1(\T_h)$, there exists a constant $C_P>0$ independent of mesh size $h$ and $v$ such that for any $p\in [1,\infty)$
\begin{align*}
\|v\|_{L^p(\Omega)}\leq C_P \|v\|_{dg},
\end{align*}
with the norm $\|\cdot\|_{dg}$ is defined by
\begin{align}\label{dg-norm}
\|v\|_{dg} =\{\|\nabla v\|_{L^2(\T_h)}^2+\sum_{e\in\E_h}h_e^{-1}\|\ljump{v}\rjump\|_{L^2(e)}^2\}^{1/2},
\end{align}
where $\ljump{v}\rjump$ is the value jump of the function $v$ across an interior edge $e$, which is defined as: if $e$ is an interior edge shared by two elements $K_1$ and $K_2$, i.e. $e=\partial{K}_1\cap\partial{K}_2$, we set
$\ljump{v}\rjump|_e = v|_{K_1}-v|_{K_2}$.
In case $e$ is a boundary edge, i.e., $e=\partial K\cap \partial\Omega$, we define $\ljump{v}\rjump|_e = v|_{K}$.
\end{lemma}

From Lemma \ref{lem:P-dg}, we can easily establish the analogue of Poincar\'{e} type inequality for weak functions.
\begin{lemma}\label{lem:P-wg}
For any weak function $v=\{v_0, v_b\}\in V_h^0+H^1(\T_h)$, there exists a constant $C_P>0$ independent of mesh size $h$ and $v$ such that for any $p\in [1,\infty)$
\begin{align*}
\|v_0\|_{L^p(\Omega)}\leq C_P \|v\|_{1,h}.
\end{align*}
\end{lemma}
\begin{proof}
Let $v=\{v_0, v_b\}$ be any weak function in the space $V_h^0+H^1(\T_h)$.
For any edge $e=\partial K_1\cap \partial K_2$, since
\begin{align*}
\ljump v_0 \rjump|_e = (v_0|_{K_1}-v_b)+(v_b-v_0|_{K_2}),
\end{align*}
then by the triangle inequality we deduce that
\begin{align*}
\|\ljump v_0 \rjump\|_{L^2(e)} \leq  \|v_0-v_b\|_{L^2(\partial K_1\cap e)}+\|v_b-v_0\|_{L^2(\partial K_2\cap e)}.
\end{align*}
Thus, under the assumption of shape regularity of the partition $\T_h$, we have
\begin{align*}
h_e^{-1}\|\ljump v_0 \rjump\|_{L^2(e)}^2 \leq  C(h_{K_1}^{-1}\|v_0-v_b\|_{L^2(\partial K_1\cap e)}^2+h_{K_2}^{-1}\|v_b-v_0\|_{L^2(\partial K_2\cap e)}^2).
\end{align*}
From the definitions of $\|\cdot\|_{dg}$ and $\|\cdot\|_{1,h}$, it is easy to see that
\begin{align*}
\|v_0\|_{dg}\leq C\|v\|_{1,h},
\end{align*}
which together with Lemma \ref{lem:P-dg} completes the proof.
\end{proof}

Let us introduce the projection operators $\mathcal{Q}_h$ and $\Pi_h$.
For each element $K\in\T_h$, we denote by $\mathcal{Q}_0$ and $\Pi_h$ the $L^2$-orthogonal projections $\mathcal{Q}_0: L^2(K)\rightarrow \mathbb{P}_k(K)$ and $\Pi_h: [L^2(K)]^2\rightarrow [\mathbb{P}_{k-1}(K)]^2$, respectively, that is,
\begin{align*}
%\label{Q0}
(\mathcal{Q}_0 v -v, w)_K &= 0,\quad \forall w\in \mathbb{P}_k(K),\\
%\label{Pih}
(\Pi_h \bm{\sigma} -\bm{\sigma}, \bm{\tau} )_K &= 0, \quad \forall \bm{\tau}\in [\mathbb{P}_{k-1}(K)]^2(K).
\end{align*}
For each edge $e\in\E_h$, the $L^2$-orthogonal projection
$\mathcal{Q}_b: L^{2}(e)\rightarrow \mathbb{P}_k(e)$ is defined as follows:
\begin{equation}\label{Qb}
\langle\mathcal{Q}_bv -v, w\rangle_e = 0,\quad \forall w\in \mathbb{P}_k(e).
\end{equation}
Finally, for a smooth function $v\in H^1(\T_h)$, we introduce the projection $\mathcal{Q}_h:H^1(\T_h)\rightarrow V_h$ such that
\begin{align*}
(\mathcal{Q}_h v)|_K: = \{\mathcal{Q}_0 (v|_K), \mathcal{Q}_b(v|_{\partial K})\}.
\end{align*}

In \cite{wg-poly}, it was proved that the projections $\mathcal{Q}_h$ and $\Pi_h$ has the following commutative property.
\begin{lemma}\citep[Lemma 5.1]{wg-poly}\label{lem:com}
For any $v\in H^1(K)$ and $K\in\T_h$, there holds
\begin{align*}
\nabla_w(\mathcal{Q}_h v) = \Pi_h(\nabla v).
\end{align*}
\end{lemma}

It is well known that (cf.\cite{ciarlet}) the projections $\mathcal{Q}_0$ and $\Pi_h$ have the standard approximation properties:
\begin{lemma}
For any $K\in\T_h$ and any $w\in W_p^{k+1}(K)$ and $\bm{\sigma}\in [W_p^{k}(K)]^2$ with $1\leq p\leq q\leq \infty$, there exists positive constant $C$ independent of $h_K$ such that,
\begin{align}\label{Q0-app}
|w-\mathcal{Q}_0 w|_{W_p^{r}(K)} \leq Ch_K^{k+1-r}|w|_{W_p^{k+1}(K)},\quad 0\leq r\leq k+1
\end{align}
and
\begin{align}\label{Pih-app}
|\bm{\sigma}-\Pi_h\bm{\sigma}|_{W_q^r(K)}\leq Ch_K^{k-r-2/p+2/q}|\bm{\sigma}|_{W_p^{k}(K)},\quad 0\leq r\leq k.
\end{align}
\end{lemma}

In the analysis below, we shall use the following Taylor series expansions.  For any $s, t\in \mathbb{R}$, let $\eta(\sigma)=a(t+\sigma(s-t))$, it is easy to check that
\[\eta(1)=\eta(0)+\int_0^1\eta'(\sigma)\mathrm{d}\sigma,\]
which implies
\begin{align}\label{eq:ts-1st}
a(s)=a(t)+\widetilde{a}_u(s,t)(s-t),
\end{align}
where $\widetilde{a}_u(s,t)=\int_0^1a_u(t+\sigma(s-t))\mathrm{d}\sigma$.
From the integral equality
\[\eta(1)=\eta(0)+\eta'(0)+\int_0^1(1-\sigma)\eta''(\sigma)\mathrm{d}\sigma,\]
it follows that
\begin{align}\label{eq:ts-2nd}
a(s)=a(t)+a_u(t)(s-t)+\widetilde{a}_{uu}(s,t)(s-t)^2,
\end{align}
where $\widetilde{a}_{uu}(s,t)=\int_0^1(1-\sigma)a_{uu}(t+\sigma(s-t))\mathrm{d}\sigma$.

Denote by
\[ M_a = \max\{\|a\|_{L^{\infty}(\Omega\times\mathbb{R})}, \|a_u\|_{L^{\infty}(\Omega\times\mathbb{R})},  \|a_{uu}\|_{L^{\infty}(\Omega\times\mathbb{R})}\}.\]

For the sake of convenience, denote by $Lu:= -\nabla\cdot(a(u)\nabla u)$. For any $w\in W^{1,p}$, the linearized operator $L$ at $w$ (namely, the Fr\'{e}chet derivative of $L$ at $w$) is then given by
\begin{align*}
L'[w]\phi := -\nabla\cdot (a(w)\nabla\phi+a_u(w)\nabla w\phi).
\end{align*}
Introducing the bilinear form (induced by $L'[w]$)
\begin{align}
D_h(w;\phi_h, v_h) = (a(w)\nabla_w \phi_h, \nabla_w v_h)_{\T_h}
+(a_u(w)\nabla w \phi_0, \nabla_wv_h)_{\T_h}+s_h(\phi_h, v_h)
\end{align}
for any $\phi_h=\{\phi_0, \phi_b\}, v_h=\{v_0, v_b\}\in V_h$. It is easy to see that
\[D_h(w;\phi_h, v_h)= A_h(w;\phi_h, v_h)+(a_u(w)\nabla w \phi_0, \nabla_wv_h)_{\T_h}.\]

First of all, let us introduce the following analogy of G{\aa}rding's inequality.
\begin{lemma}[G{\aa}rding's inequality]\label{lem:garding}
For a given $\phi\in W^1_{\infty}(\Omega)$. Then there is a positive constant $\beta$ satisfying
\begin{align}\label{beta-con}
 \gamma+\frac{M_a^2|\phi|_{W^1_{\infty}(\Omega)}^2}{2\alpha_0}
 \leq \beta < \infty,
\end{align}
such that
\begin{equation}\label{eq:garding}
D_h(\phi;v_h,v_h)+\beta\|v_0\|^2 \geq \gamma (\3bar v_h\3bar^2+\|v_0\|^2), \quad \forall v_h=\{v_0,v_b\}\in V_h^0,
\end{equation}
where $\gamma = \min\{\frac{\alpha_0}{2},1\}$.
\end{lemma}
\begin{proof}
Using the boundness of $a(u)$, we have
\begin{align*}
&D_h(\phi;v_h,v_h)+\beta\|v_0\|_{L^2(\T_h)}^2\\
&=A_h(\phi;v_h, v_h)+(a_u(\phi)\nabla \phi v_0, \nabla_wv_h)_{\T_h}+\beta\|v_0\|_{L^2(\T_h)}^2\\
&\geq \alpha_0\|\nabla_w v_h\|_{L^2(\T_h)}^2+s_h(v_h,v_h) + (a_u(\phi)\nabla \phi v_0, \nabla_wv_h)_{\T_h}
+\beta\|v_0\|_{L^2(\T_h)}^2.
\end{align*}
By H\"{o}lder's inequality,
\begin{align*}
|(a_u(\phi)\nabla \phi v_0, \nabla_wv_h)_{\T_h}|
\leq M_a| \phi|_{W^1_{\infty}(\Omega)}\|\nabla_w v_h\|_{L^2(\T_h)}\|v_0\|_{L^2(\T_h)}.
\end{align*}
Therefore,
\begin{align*}
&D_h(\phi;v_h,v_h)+\beta\|v_0\|_{L^2(\T_h)}^2\\
&\geq \alpha_0\|\nabla_w v_h\|_{L^2(\T_h)}^2+s_h(v_h,v_h)+\beta\|v_0\|_{L^2(\T_h)}^2\\
&\quad -M_a| \phi|_{W^1_{\infty}(\Omega)}\|\nabla_w v\|_{L^2(\T_h)}\|v_0\|_{L^2(\T_h)}.
\end{align*}
Let $\gamma = \min\{\frac{\alpha_0}{2},1\}$. Provided (\ref{beta-con}),
from the Young's inequality, we conclude that
\begin{align*}
&D_h(\phi;v_h,v_h)+\beta\|v_0\|_{L^2(\T_h)}^2\\
&\geq \frac{\alpha_0}{2}\|\nabla_w v_h\|_{L^2(\T_h)}^2+s_h(v_h,v_h)
+(\beta-\frac{M_a^2|\phi|_{W^1_{\infty}(\Omega)}^2}{2\alpha_0})\|v_0\|_{L^2(\T_h)}^2\\
&\geq\gamma(\3bar v_h\3bar^2+\|v_0\|_{L^2(\T_h)}^2).
\end{align*}
The proof is completed.
\end{proof}

\subsection{Error equation}
Now for the exact solution $u$ of (\ref{pde})-(\ref{pde-bc}), we define the error between the WG solution $u_h=\{u_0, u_b\}$ and the projection $\mathcal{Q}_hu=\{\mathcal{Q}_0u, \mathcal{Q}_bu\}$ of $u$ as
\[
 e_h=\mathcal{Q}_hu-u_h:=\{e_0, e_b\},
\]
with
\[
e_0=\mathcal{Q}_0u-u_0, \quad e_b=\mathcal{Q}_bu-u_b.
\]
The aim of this subsection is to obtain an error equation for $e_h$ by the use of bilinear form $D_h(u;\cdot,\cdot)$.

\begin{lemma}[Error equation]
Let $u$ and $u_h$ be the solutions of the problem (\ref{pde})-(\ref{pde-bc}) and the WG scheme (\ref{wg}), respectively.
For any $v=\{v_0, v_b\}\in V_h^0$, we have
\begin{align}\label{eq:dh}
D_h(u; \mathcal{Q}_hu-u_h,v)=E_h(u, v)+R_h(u; u_h, v),
\end{align}
where
\begin{align}\label{eq:Eh}
E_h(u, v)&=\ell(u, v) + s_h(\mathcal{Q}_hu, v),\\[3pt]
\label{eq:Rh}
R_h(u; u_h, v)&= ((u-u_0)\widetilde{a}_u(u_0)(\nabla u-\nabla_wu_h), \nabla_w v)_{\T_h}\nonumber\\
&\qquad +((u-u_0)^2\widetilde{a}_{uu}(u_0){\nabla u}, \nabla_w v)_{\T_h},
\end{align}
with $\ell(u,v):=\sum\nolimits_{i=1}^3\ell_i(u,v)$, and
\begin{subequations}
  \begin{align}
  \label{el1}
  \ell_1(u,v)&:=(a(u)(\Pi_h(\nabla u)-\nabla u), \nabla_w v)_{\T_h},\\
  \label{el2}
  \ell_2(u,v)&:= \langle (a(u)\nabla u-\Pi_h(a(u)\nabla u))\cdot\bm{n}, v_0-v_b
  \rangle_{\partial\T_h},\\
  \label{el3}
  \ell_3(u,v)&:= ((\mathcal{Q}_0u-u)a_u(u)\nabla{u}, \nabla_w v)_{\T_h}.
  \end{align}
\end{subequations}
\end{lemma}
\begin{proof}
For notational convenience, we denote $a(u)\nabla u$ by $\sigma$.
For any $v=\{v_0,v_b\}\in V_h^0$, testing (\ref{pde}) by  $v_0$  and using the fact that
\[\sum_{K\in\T_h}\langle \sigma\cdot\bm{n}, v_b\rangle_{\partial K}=0\]
and integration by parts,  we arrive at
\begin{align}\label{eq:m1}
(f,v_0)=-(\nabla\cdot\sigma, v_0)_{\T_h}
%&=&(\sigma,\nabla v_0)_{\T_h} -\langle \sigma\cdot\bm{n},
%v_0\rangle_{\partial\T_h}\nonumber\\
=(\sigma,\nabla v_0)_{\T_h} -\langle \sigma\cdot\bm{n},
v_0-v_b\rangle_{\partial\T_h}.
\end{align}
Next we investigate the term  $(\sigma,\nabla v_0)_{\T_h}$ in the above equation.
 Using the definition of $\Pi_h$, integration by parts and the definition of weak gradient, we have
\begin{eqnarray*}
 (\sigma,\nabla v_0)_{\T_h}&=&(\Pi_h\sigma,\nabla v_0)_{\T_h} \\
&=& -(v_0, \nabla\cdot \Pi_h\sigma )_{\T_h}
    +\langle v_0, \Pi_h\sigma\cdot\bm{n} \rangle_{\partial\T_h}\nonumber\\
&=&(\nabla_w v, \Pi_h\sigma)_{\T_h}
  +\langle v_0-v_b, \Pi_h\sigma\cdot\bm{n}
\rangle_{\partial\T_h}\nonumber\\
&=&(\sigma, \nabla_w v)_{\T_h}
  +\langle \Pi_h\sigma\cdot\bm{n}, v_0-v_b
\rangle_{\partial\T_h},
\end{eqnarray*}
which together with (\ref{eq:m1}) yields
\begin{align}\label{eq:m4}
(\sigma, \nabla_w v)_{\T_h}
=(f,v_0)
+\langle (\sigma-\Pi_h\sigma)\cdot\bm{n}, v_0-v_b
\rangle_{\partial\T_h}.
\end{align}
Recalling Lemma \ref{lem:com} and adding $s_h(\mathcal{Q}_hu, v)$ on both sides of (\ref{eq:m4}), we arrive at
\begin{eqnarray*}
(a(u)\nabla_w(\mathcal{Q}_hu), \nabla_wv)_{\T_h}
+s_h(\mathcal{Q}_hu, v)=(f,v_0)+\sum_{i=1}^2\ell_i(u,v)+s_h(\mathcal{Q}_hu, v).
\end{eqnarray*}

Subtracting the WG scheme (\ref{wg}) from the above equation  yields
\begin{align*}
A_h(u;\mathcal{Q}_hu-u_h, v) = \sum_{i=1}^2\ell_i(u, v) + s_h(\mathcal{Q}_hu, v)
+((a(u_0)-a(u))\nabla_wu_h, \nabla_w v)_{\T_h},
\end{align*}
then adding the term $((\mathcal{Q}_0u-u_0)a_u(u)\nabla{u}, \nabla_w v)_{\T_h}$ on both sides gives
\begin{align}\label{eq:tmp1}
D_h(u;\mathcal{Q}_hu-u_h, v) = E_h(u,v)+\mathfrak{T},
\end{align}
with $E_h(u,v)$ is defined by (\ref{eq:Eh}), and
\begin{align}\label{eq:H}
\mathfrak{T}&=((a(u_0)-a(u))\nabla_wu_h, \nabla_w v)_{\T_h}+((u-u_0)a_u(u)\nabla{u}, \nabla_w v)_{\T_h}\nonumber\\
&=((a(u_0)-a(u))(\nabla_wu_h-\nabla u), \nabla_w v)_{\T_h}\nonumber\\
&\quad +((a(u_0)-a(u)-(u_0-u)a_u(u))\nabla{u}, \nabla_w v)_{\T_h}.
\end{align}
Using the Taylor expansions (\ref{eq:ts-1st}) and (\ref{eq:ts-2nd}), we get
\begin{align*}
\mathfrak{T}= ((u-u_0)\widetilde{a}_u(u_0,u)(\nabla u-\nabla_wu_h), \nabla_w v)_{\T_h}
+((u-u_0)^2\widetilde{a}_{uu}(u_0,u){\nabla u}, \nabla_w v)_{\T_h},
\end{align*}
which together with (\ref{eq:tmp1}) completes the proof.
\end{proof}

\begin{lemma}\label{lem:l2}
Assume $u\in H^{k+1}(\Omega)\cap W_{\infty}^1(\Omega)$. There exists a constant $C$ such that the following estimates hold true:
\begin{align}
\label{eq:z1}
|\ell_1(u,v)|&\leq Ch^{k}|u|_{H^{k+1}(\Omega)}\|\nabla_w v\|_{L^2(\T_h)},\\
\label{eq:z2}
|\ell_2(u,v)|&\leq Ch^{k}|u|_{H^{k+1}(\Omega)}s_h^{1/2}(v,v),\\
\label{eq:z4}
|\ell_3(u,v)|&\leq Ch^{k+1}|u|_{H^{k+1}(\Omega)}\|\nabla_w v\|_{L^2(\T_h)},\\
\label{eq:z3}
|s_h(\mathcal{Q}_hu, v)|&\leq Ch^{k}|u|_{H^{k+1}(\Omega)}s_h^{1/2}(v,v).
\end{align}
\end{lemma}
\begin{proof}
By  the triangle inequality, H\"{o}lder's inequality and the approximation property of $\Pi_h$, we have
\begin{align}\label{eq:m2}
|\ell_1(u, v)|
&\leq \sum_{K\in\T_h}|(\Pi_h(\nabla u)-\nabla u, a(u)\nabla_w v)_{K}|\nonumber\\
&\leq \sum_{K\in\T_h}M_a\|\Pi_h(\nabla u)-\nabla u\|_{L^2(K)}\|\nabla_w{v}\|_{L^2(K)}
\nonumber\\
&\leq Ch^{k}|u|_{H^{k+1}(\Omega)}\|\nabla_w v\|_{L^2(\T_h)}.
\end{align}

For the estimate (\ref{eq:z2}), from {H\"{o}lder's} inequality, (\ref{Pih-app}) and the trace inequality (\ref{trace1}), it follows that
\begin{align*}
|\ell_2(u,v)|
&\leq \sum_{K\in\T_h}\|a(u)\nabla u-\Pi_h(a(u)\nabla u)\|_{L^2(\partial K)}
\|v_0-v_b\|_{L^2(\partial K)}\\
&\leq (\sum_{K\in\T_h}h_K\|a(u)\nabla u-\Pi_h(a(u)\nabla u)\|_{L^2(\partial K)}^2)^{1/2}\\
&\quad \times (\sum_{K\in\T_h}h_K^{-1}\|v_0-v_b\|_{L^2(\partial K)}^2)^{1/2}\\
&\leq Ch^{k}|u|_{H^{k+1}(\Omega)}s_h^{1/2}(v,v).
\end{align*}

By H\"{o}lder's inequality and (\ref{Q0-app}),   we can deduce that
\begin{align*}
|\ell_3(u, v)|&\leq \sum_{K\in \T_h}M_a\|\mathcal{Q}_0u-u\|_{L^2(K)}\|\nabla u\|_{L^{\infty}(K)}\|\nabla_w v\|_{L^2(K)}\nonumber\\
&\leq C\sum_{K\in \T_h} h_K^{k+1}|u|_{H^{k+1}(K)}|u|_{W_{\infty}^1(K)}\cdot \|\nabla_w v\|_{L^2(K)}\nonumber\\
&\leq Ch^{k+1}|u|_{H^{k+1}(\Omega)}\|\nabla_w v\|_{L^2(\Omega)}.
\end{align*}

Now we consider the estimate (\ref{eq:z3}). By the trace inequality (\ref{trace1}) and (\ref{Q0-app}), we have
\begin{align*}
\|\mathcal{Q}_0u-u\|_{L^2(\partial K)}\leq Ch_K^{k+1/2}|u|_{H^{k+1}(K)}.
\end{align*}
Then, it follows from the Cauchy-Schwarz inequality and the definition of $\mathcal{Q}_b$ that
\begin{align*}
|s_h(\mathcal{Q}_hu,v)|
&\leq \sum_{K\in\T_h} h_K^{-1}|\langle \mathcal{Q}_0u-\mathcal{Q}_bu,
v_0-v_b\rangle_{\partial K}|\nonumber\\
&=\sum_{K\in\T_h} h_K^{-1}|\langle \mathcal{Q}_0u-u,
v_0-v_b\rangle_{\partial K}|\nonumber\\
&\leq (\sum_{K\in\T_h} h_K^{-1}\|\mathcal{Q}_0u-u\|_{L^2({\partial K})}^2)^{1/2}
(\sum_{K\in\T_h} h_K^{-1}
\|v_0-v_b\|_{L^2({\partial K})}^2)^{1/2}\nonumber\\
&\leq C h^{k}|u|_{H^{k+1}(\Omega)} s_h^{1/2}(v, v).
\end{align*}
We have completed the proof.
\end{proof}

\subsection{Construction of a fixed point mapping}

Motivated by the form of the error equation (\ref{eq:dh}), we introduce the fixed point mapping $\mathcal{F}_h$ as follows.

\begin{definition}[The fixed point mapping $\mathcal{F}_h$]
For a given $\xi\in V_h$, let $\mathcal{F}_h: V_h\rightarrow V_h$ be a map $\psi=\mathcal{F}_h(\xi)\in V_h$ satisfying
\begin{align}\label{eq:f-map}
D_h(u; \mathcal{Q}_hu-\psi, v_h) = E_h(u, v_h) + R_h(u; \xi, v_h),
\end{align}
for any $v_h\in V_h^0$.
\end{definition}
\begin{remark}
Equivalently, we can restate (\ref{eq:f-map}) as follows: find $\psi\in V_h$ such that
\begin{align}\label{eq:f-map-eq}
D_h(u; \psi, v_h)=\Psi_{\xi}(v_h),
\end{align}
for any $v_h\in V_h^0$, with $\Psi_{\xi}(v)$ is defined by
\[
\Psi_{\xi}(v):=D_h(u; \mathcal{Q}_hu, v) - E_h(u, v) - R_h(u; \xi, v).
\]
Obviously, $\Psi_{\xi}(v)$ is a continuous linear functional of $v$. By Riesz representation theorem, there exists a function $\widehat{f}$ such that $(\widehat{f}, v)=\Psi_{\xi}(v)$. The solution $\psi\in V_h$ of (\ref{eq:f-map-eq}) can be viewed as a WG finite element solution of the linear second order elliptic equation $L'[u]\phi = \widehat{f}$. By the use of G{\aa}rding's inequality (\ref{eq:garding}), it is easy to show that there exists a unique solution of (\ref{eq:f-map-eq}) as the proof in \cite{wy}. Then the map $\psi=\mathcal{F}_h(\xi)$ given by (\ref{eq:f-map}) is well defined.
\end{remark}

Now consider the ball
\[
\mathbb{B}_{h}(\mathcal{Q}_hu)=\{\omega\in V_h: \3bar \mathcal{Q}_hu-\omega\3bar \leq Ch^k|u|_{H^{k+1}(\Omega)}\},
\]
where $C$ is a general positive constant. Since $\mathcal{Q}_hu\in \mathbb{B}_{h}(\mathcal{Q}_hu)$, the set $\mathbb{B}_{h}(\mathcal{Q}_hu)$ is nonempty.

In the below, we shall show that $\mathcal{F}_h$ has a fixed point in $\mathbb{B}_{h}(\mathcal{Q}_hu)$.  Since (\ref{eq:dh}) is equivalent to the WG scheme (\ref{wg}), there exists a solution $u_h\in V_h$ of the nonlinear WG finite element scheme (\ref{wg}).

\begin{lemma}\label{lem:l3}
Assume $u\in H^{k+1}(\Omega)\cap W_{\infty}^1(\Omega)$. For any $\xi\in V_h$ and $v_h\in V_h^0$, there exists a constant $C>0$, independent of $h$, such that
\begin{align}
|R_h(u;\xi, v_h)|\leq Ch^{-1/2}(h^k|u|_{H^{k+1}(\Omega)}+\3bar \mathcal{Q}_hu-\xi\3bar)^2\|\nabla_w{v_h}\|_{L^2(\T_h)}.
\end{align}
\end{lemma}
\begin{proof}
For any $\xi\in V_h$, let $\xi=\{\xi_0, \xi_b\}$. From (\ref{eq:Rh}), we know
\begin{align*}
R_h(u;\xi, v_h)&=((u-\xi_0)\widetilde{a}_u(\xi_0,u)(\nabla u-\nabla_w \xi), \nabla_w{v_h})_{\T_h}\nonumber\\
&\quad +((u-\xi_0)^2\widetilde{a}_{uu}(\xi_0,u)\nabla u, \nabla_w{v_h})_{\T_h}.
\end{align*}
By H\"{o}lder's inequality and the inverse inequality (\ref{inv}), we can deduce that
\begin{align}\label{eq:h1}
&|((u-\xi_0)\widetilde{a}_u(\xi_0,u)(\nabla u-\nabla_w \xi), \nabla_w{v_h})_{\T_h}|\nonumber\\[2pt]
&\leq M_a\|u-\xi_0\|_{L^4(\T_h)}\|\nabla u-\nabla_w \xi\|_{L^2(\T_h)}\|\nabla_w{v_h}\|_{L^4(\T_h)}\nonumber\\
&\leq Ch^{-1/2}\|u-\xi_0\|_{L^4(\T_h)}\|\nabla u-\nabla_w\xi\|_{L^2(\T_h)}
\|\nabla_w{v_h}\|_{L^2(\T_h)},
\end{align}
and
\begin{align}\label{eq:h2}
&|((u-\xi_0)^2\widetilde{a}_{uu}(\xi_0,u)\nabla u, \nabla_w{v_h})_{\T_h}|\nonumber\\
&\leq M_a\|(u-\xi_0)^2\|_{L^2(\T_h)}\|\nabla u\|_{L^{\infty}(\Omega)}\|\nabla_w{v_h}\|_{L^2(\T_h)}\nonumber\\
&\leq C\|u-\xi_0\|_{L^4(\T_h)}^2\|\nabla_w{v_h}\|_{L^2(\T_h)}.
\end{align}

Thanks to Lemma \ref{lem:P-wg}, we get
\begin{align*}
\|u-\xi_0\|_{L^4(\T_h)}\leq C_P \|u-\xi\|_{1,h}\leq C_P(\|u-\mathcal{Q}_0u\|_{1,h}+\|\mathcal{Q}_hu-\xi\|_{1,h}),
\end{align*}
then using (\ref{Q0-app}), the trace inequality (\ref{trace1}) and Lemma \ref{lem:happy}, we obtain
\begin{align}\label{eq:h1-1}
\|u-\xi_0\|_{L^4(\T_h)}\leq C(h^k|u|_{H^{k+1}(\Omega)}+\3bar \mathcal{Q}_hu-\xi\3bar).
\end{align}

By the triangle inequality, (\ref{Pih-app}) and Lemma \ref{lem:com}, we have
\begin{align}\label{eq:h2-2}
\|\nabla u-\nabla_w \xi\|_{L^2(\T_h)}
&\leq \|\nabla u-\Pi_h(\nabla u)\|_{L^2(\T_h)}
+\|\nabla_w (\mathcal{Q}_h u- \xi)\|_{L^2(\T_h)}\nonumber\\
&\leq C(h^k|u|_{H^{k+1}(\Omega)}+\3bar \mathcal{Q}_hu-\xi\3bar).
\end{align}
Plugging (\ref{eq:h1-1}) and (\ref{eq:h2-2}) back into (\ref{eq:h1}) and (\ref{eq:h2}), respectively, completes the proof.
\end{proof}

\begin{lemma}\label{lem:l4}
Assume $u\in H^{k+1}(\Omega)\cap W_{\infty}^1(\Omega)$. For any $\xi\in \mathbb{B}_{h}(\mathcal{Q}_hu)$, let $\psi=\mathcal{F}_h(\xi)$ and denote by $\zeta_h=\mathcal{Q}_hu-\psi$. Then there exists a constant $C>0$, independent of $h$, such that
\begin{align}\label{eq:zeta-3bar}
\gamma\3bar\zeta_h\3bar
\leq \beta\|\zeta_0\|_{L^2(\T_h)}+Ch^k|u|_{H^{k+1}(\Omega)},
\end{align}
for sufficiently small $h$.
\end{lemma}
\begin{proof}
Taking $v_h=\zeta_h$ in (\ref{eq:f-map}) yields
\begin{align}\label{eq:f1}
D_h(u;\zeta_h, \zeta_h) = E_h(u, \zeta_h) + R_h(u;\xi, \zeta_h).
\end{align}
From Lemma \ref{lem:l2}, we can see that
\begin{align}\label{eq:f2}
|E_h(u, \zeta_h)|\leq Ch^k|u|_{H^{k+1}(\Omega)}\3bar\zeta_h\3bar.
\end{align}
In view of Lemma \ref{lem:l3} and noting that $\xi\in \mathbb{B}_{h}(\mathcal{Q}_hu)$, we arrive at
\begin{align*}
|R_h(u;\xi, \zeta_h)|\leq Ch^{2k-1/2}|u|_{H^{k+1}(\Omega)}^2\3bar\zeta_h\3bar,
\end{align*}
which combining with (\ref{eq:f1}) and (\ref{eq:f2}) leads to
\begin{align*}
D_h(u;\zeta_h, \zeta_h)&\leq C(1+h^{k-1/2}|u|_{H^{k+1}(\Omega)})h^k|u|_{H^{k+1}(\Omega)}\3bar\zeta_h\3bar\\
&\leq Ch^k|u|_{H^{k+1}(\Omega)}\3bar\zeta_h\3bar,
\end{align*}
for sufficiently small $h$.

Appealing to the G{\aa}rding's inequality (\ref{eq:garding}) yields
\begin{align*}
\gamma (\3bar\zeta_h\3bar^2+\|\zeta_0\|_{L^2(\T_h)}^2)
&\leq \beta\|\zeta_0\|_{L^2(\T_h)}^2+D_h(u;\zeta_h, \zeta_h)\\
&\leq \beta\|\zeta_0\|_{L^2(\T_h)}^2+Ch^k|u|_{H^{k+1}(\Omega)}\3bar\zeta_h\3bar,
\end{align*}
which implies
\begin{align*}
\gamma\3bar\zeta_h\3bar
\leq \beta\|\zeta_0\|_{L^2(\T_h)}+Ch^k|u|_{H^{k+1}(\Omega)}.
\end{align*}
The proof is completed.
\end{proof}

In order to bound $\|\zeta_0\|_{L^2(\T_h)}$ we use a duality argument.
Consider the following dual problem
\begin{align}\label{dual}
\left\{\begin{array}{rr}
-\nabla\cdot(a(u)\nabla \varphi)+a_u(u)\nabla u\cdot \nabla\varphi= \zeta_0&\quad
\mbox{in}\;\;\Omega,\\[3pt]
\varphi=0&\quad\mbox{on}\;\partial\Omega.
\end{array}\right.
\end{align}
Assume that the dual problem has the $H^{2}$-regularity in the sense that there exists a constant $C$ such that
\begin{equation}\label{reg}
\|\varphi\|_{H^{2}(\Omega)}\leq C\|\zeta_0\|_{L^2(\T_h)}.
\end{equation}

\begin{lemma}\label{lem:l6}
Assume $u\in H^{k+1}(\Omega)\cap W_{\infty}^1(\Omega)$ and $\varphi\in H^{2}(\Omega)$. There exists a constant $C$ such that
\begin{align*}
|E_h(u,\mathcal{Q}_h\varphi)|&\leq Ch^{k+1}|u|_{H^{k+1}(\Omega)}\| \varphi\|_{H^2(\Omega)}.
\end{align*}
\end{lemma}
\begin{proof}
By  the triangle inequality, H\"{o}lder's inequality and the approximation property of $\Pi_h$, we have
\begin{align}\label{eq:lem3-eq1}
|\ell_1(u, \mathcal{Q}_h\varphi)|
&\leq \sum_{K\in\T_h}|(\Pi_h(\nabla u)-\nabla u, a(u)\nabla_w (\mathcal{Q}_h\varphi))_{K}|\nonumber\\
&\leq \sum_{K\in\T_h}|(\Pi_h(\nabla u)-\nabla u, a(u)(\Pi_h (\nabla\varphi)-\nabla\varphi))_{K}|\nonumber\\
&\quad +\sum_{K\in\T_h}|(\Pi_h(\nabla u)-\nabla u, a(u)\nabla\varphi-\Pi_h (a(u)\nabla\varphi))_{K}|\nonumber\\
&\leq \sum_{K\in\T_h}\|\Pi_h(\nabla u)-\nabla u\|_{L^2(K)}(M_a\|\Pi_h (\nabla\varphi)-\nabla\varphi\|_{L^2(K)}\nonumber\\
&\qquad\qquad +\|a(u)\nabla\varphi-\Pi_h (a(u)\nabla\varphi)\|_{L^2(K)})\nonumber\\
&\leq Ch^{k+1}|u|_{H^{k+1}(\Omega)}|\varphi|_{H^2(\Omega)}.
\end{align}

First of all, we get the bound of $s_h(\mathcal{Q}_h\varphi, \mathcal{Q}_h\varphi)$. By the definition of $\mathcal{Q}_b$, it is easy to see that
\[ \|\mathcal{Q}_b\varphi-\varphi\|_{L^2(\partial K)}\leq \|\mathcal{Q}_0\varphi-\varphi\|_{L^2(\partial K)},\]
then using the triangle inequality, we have
\[ \|\mathcal{Q}_0\varphi-\mathcal{Q}_b\varphi\|_{L^2(\partial K)}\leq 2\|\mathcal{Q}_0\varphi-\varphi\|_{L^2(\partial K)}.\]
Thus, from the Cauchy-Schwarz inequality, (\ref{Q0-app}) and the trace inequality (\ref{trace1}), it follows that
\begin{align}\label{eq:sh}
s_h(\mathcal{Q}_h\varphi, \mathcal{Q}_h\varphi)
&= \sum_{K\in\T_h}h_K^{-1}\|\mathcal{Q}_0\varphi-\mathcal{Q}_b\varphi\|_{L^2(\partial K)}^2\nonumber\\
&\leq 4\sum_{K\in\T_h}h_K^{-1}\|\mathcal{Q}_0\varphi-\varphi\|_{L^2(\partial K)}^2\nonumber\\
&\leq Ch^2|\varphi|_{H^2(\Omega)}^2.
\end{align}

Then, using (\ref{eq:z2}) of Lemma \ref{lem:l2} and (\ref{eq:sh}), we obtain
\begin{align}\label{eq:l2-est}
|\ell_2(u, \mathcal{Q}_h\varphi)|\leq Ch^k|u|_{H^{k+1}(\Omega
)} s_h^{1/2}(\mathcal{Q}_h\varphi, \mathcal{Q}_h\varphi)
\leq Ch^{k+1}|u|_{H^{k+1}(\Omega
)}|\varphi|_{H^2(\Omega)}.
\end{align}
Similarly, we have
\begin{align}\label{eq:sh2}
|s_h(\mathcal{Q}_hu, \mathcal{Q}_h\varphi)|\leq  Ch^{k+1}|u|_{H^{k+1}(\Omega
)}|\varphi|_{H^2(\Omega)}.
\end{align}

By the H\"{o}lder's inequality and (\ref{Q0-app}),   we can deduce that
\begin{align}\label{eq:lem3-eq2}
|\ell_3(u, \mathcal{Q}_h\varphi)|&\leq \sum_{K\in \T_h}M_a\|\mathcal{Q}_0u-u\|_{L^2(K)}\|\nabla u\|_{L^{\infty}(K)}\|\nabla_w(\mathcal{Q}_h\varphi)\|_{L^2(K)}\nonumber\\
&\leq C\sum_{K\in \T_h} h_K^{k+1}|u|_{H^{k+1}(K)}|u|_{W_{\infty}^1(K)}\cdot \|\Pi_h(\nabla\varphi)\|_{L^2(K)}\nonumber\\
&\leq Ch^{k+1}|u|_{H^{k+1}(\Omega)}|\varphi|_{H^1(\Omega)}.
\end{align}

With the help of the estimates (\ref{eq:lem3-eq1}), (\ref{eq:l2-est}), (\ref{eq:lem3-eq2}) and (\ref{eq:sh2}), we completes the proof.
\end{proof}

{
We need the following two lemmas to prove Lemma \ref{lem:l7}.
\begin{lemma}\cite[Lemma 2.3]{qiu}
Let $v\in L^p(K)$ with $p\geq 1$ and $K\in\T_h$. Then, we have
\begin{align}\label{Lp-stab}
\|\Pi_hv\|_{L^p(K)}\leq C\|v\|_{L^p(K)},
\end{align}
where $C$ is independent of the diameter $h_K$ of $K$.
\end{lemma}
\begin{lemma}\cite[Theorem 5.8]{adams}
For any $v\in H^1(\Omega)$, there holds
\begin{align}\label{sobolev-ineq}
\| v\|_{L^4(\Omega)}\leq C\| v\|_{L^2(\Omega)}^{1/2}\|v\|_{H^1(\Omega)}^{1/2}.
\end{align}
\end{lemma}
}
\begin{lemma}\label{lem:l7}
Assume $u\in H^{k+1}(\Omega)\cap W_{\infty}^1(\Omega)$ and $\varphi\in H^{2}(\Omega)$. For any $\xi\in \mathbb{B}_{h}(\mathcal{Q}_hu)$, there exists a constant $C>0$, independent of $h$, such that
\begin{align*}
|R_h(u;\xi, \mathcal{Q}_h\varphi)|\leq Ch^{2k}|u|_{H^{k+1}(\Omega)}^2\|\varphi\|_{H^2(\Omega)}.
\end{align*}
\end{lemma}
\begin{proof}
By Lemma \ref{lem:com}, $\nabla_w(\mathcal{Q}_h\varphi)=\Pi_h(\nabla \varphi)$.
Then from (\ref{eq:Rh}), we know
\begin{align*}
R_h(u;\xi, \mathcal{Q}_h\varphi)&=((u-\xi_0)\widetilde{a}_u(\xi_0,u)(\nabla u-\nabla_w \xi), \Pi_h(\nabla\varphi))_{\T_h}\nonumber\\
&\quad +((u-\xi_0)^2\widetilde{a}_{uu}(\xi_0,u)\nabla u, \Pi_h(\nabla\varphi))_{\T_h}.
\end{align*}
By H\"{o}lder's inequality and the inverse inequality (\ref{inv}), we can deduce that
\begin{align}\label{eq:lem7-eq1}
&|((u-\xi_0)\widetilde{a}_u(\xi_0,u)(\nabla u-\nabla_w \xi), \Pi_h(\nabla\varphi))_{\T_h}|\nonumber\\[2pt]
&\leq M_a\|u-\xi_0\|_{L^4(\T_h)}\|\nabla u-\nabla_w \xi\|_{L^2(\T_h)}\|\Pi_h(\nabla\varphi)\|_{L^4(\T_h)},
\end{align}
and
\begin{align}\label{eq:lem7-eq2}
&|((u-\xi_0)^2\widetilde{a}_{uu}(\xi_0,u)\nabla u, \Pi_h(\nabla\varphi))_{\T_h}|\nonumber\\
&\leq M_a\|(u-\xi_0)^2\|_{L^2(\T_h)}\|\nabla u\|_{L^{\infty}(\Omega)}\|\Pi_h(\nabla\varphi)\|_{L^2(\T_h)}\nonumber\\
&\leq C\|u-\xi_0\|_{L^4(\T_h)}^2\|\Pi_h(\nabla\varphi)\|_{L^2(\T_h)}.
\end{align}
{
In virtue of (\ref{Lp-stab}) and (\ref{sobolev-ineq}), we obtain
\begin{align}\label{eq:0}
\|\Pi_h(\nabla\varphi)\|_{L^4(\T_h)} \leq C\|\nabla\varphi\|_{L^4(\Omega)}
\leq C\|\nabla \varphi\|_{L^2(\Omega)}^{1/2}\|\nabla \varphi\|_{H^1(\Omega)}^{1/2}
\leq C\|\varphi\|_{H^2(\Omega)}.
\end{align}
}

From the triangle inequality and (\ref{Pih-app}), it follows that
\begin{align}\label{eq:lem7-eq3}
\|\Pi_h(\nabla \varphi)\|_{L^2(\T_h)}
&\leq \|\Pi_h(\nabla \varphi)-\nabla \varphi\|_{L^2(\T_h)}+\|\nabla \varphi\|_{L^2(\T_h)}\nonumber\\
&\leq Ch|\varphi|_{H^2(\Omega)}+|\varphi|_{H^1(\Omega)}\nonumber\\
&\leq C\|\varphi\|_{H^2(\Omega)}.
\end{align}

Plugging (\ref{eq:h1-1}), (\ref{eq:h2-2}), (\ref{eq:0})  and (\ref{eq:lem7-eq3}) back into (\ref{eq:lem7-eq1}) and (\ref{eq:lem7-eq2}), respectively, completes the proof.
\end{proof}

\begin{lemma}\label{lem:l5}
Assume $u\in H^{k+1}(\Omega)\cap W_{\infty}^{1}(\Omega)$. For any $\xi\in \mathbb{B}_{h}(\mathcal{Q}_hu)$, let $\psi=\mathcal{F}_h(\xi)$ and denote by $\zeta_h=\mathcal{Q}_hu-\psi$. Then there exists a constant $C>0$, independent of $h$, such that
\begin{equation}\label{eq:zeta-l2}
\|\zeta_0\|_{L^2(\T_h)}
\leq C[h\3bar\zeta_h\3bar+h^{k+1}|u|_{H^{k+1}(\Omega)}+ h^{2k}|u|_{H^{k+1}(\Omega)}^2],
\end{equation}
for sufficiently small $h$.
\end{lemma}
\begin{proof}
Testing the first equation of the dual problem (\ref{dual}) by $\zeta_0$ yields
\begin{align*}
\|\zeta_0\|_{L^2(\Omega)}^2
=-(\nabla\cdot(a(u)\nabla\varphi), \zeta_0)_{\T_h}+(\zeta_0a_u(u)\nabla u, \nabla\varphi)_{\T_h}.
\end{align*}
Then using a similar procedure as in the proof of the equation (\ref{eq:m4}), we obtain
\begin{align*}
&-(\nabla\cdot(a(u)\nabla\varphi), \zeta_0)_{\T_h}\nonumber\\
&=(a(u)\nabla_w\zeta_h, \nabla\varphi)_{\T_h}
 +\langle \zeta_0-\zeta_b, (\Pi_h(a(u)\nabla\varphi)-a(u)\nabla\varphi)\cdot\bm{n}\rangle_{\partial\T_h}.
\end{align*}
Therefore,
\begin{align*}
\|\zeta_0\|_{L^2(\Omega)}^2
&=(a(u)\nabla_w\zeta_h, \nabla\varphi)_{\T_h}
+(\zeta_0a_u(u)\nabla u, \nabla\varphi)_{\T_h}\\
&\quad +\langle \zeta_0-\zeta_b, (\Pi_h(a(u)\nabla\varphi)-a(u)\nabla\varphi)\cdot\bm{n}\rangle_{\partial\T_h}.
\end{align*}
With the following notations
\begin{align}
\label{eq:J1}
J_1 &= \langle \zeta_0-\zeta_b, (\Pi_h(a(u)\nabla\varphi)-a(u)\nabla\varphi)\cdot\bm{n}\rangle_{\partial\T_h},\\[2pt]
\label{eq:J2}
J_2 &= (a(u)(\nabla\varphi-\Pi_h(\nabla\varphi)), \nabla_w\zeta_h)_{\T_h}
,\\[2pt]
\label{eq:J3}
J_3 &= (a_u(u)\nabla u\zeta_0, \nabla\varphi-\Pi_h(\nabla\varphi))_{\T_h},
\end{align}
we could rewrite the term $\|\zeta_0\|_{L^2(\Omega)}^2$ as follows:
\begin{align}\label{eq:lem4-eq1}
\|\zeta_0\|_{L^2(\Omega)}^2
&=(a(u)\nabla_w\zeta_h, \nabla_w(\mathcal{Q}_h\varphi))_{\T_h}
+(\zeta_0a_u(u)\nabla u, \nabla_w(\mathcal{Q}_h\varphi))_{\T_h}\nonumber\\
&\quad +J_1+J_2+J_3\nonumber\\
&=D_h(u;\zeta_h, \mathcal{Q}_h\varphi)-s_h(\zeta_h, \mathcal{Q}_h\varphi)
+J_1+J_2+J_3.
\end{align}
Taking $v=\mathcal{Q}_h\varphi$ in (\ref{eq:f-map}) yields
\begin{align*}
D_h(u;\zeta_h, \mathcal{Q}_h\varphi)
=E_h(u,\mathcal{Q}_h\varphi)
+R_h(u;\xi,\mathcal{Q}_h\varphi),
\end{align*}
which together with (\ref{eq:lem4-eq1}) leads to
\begin{align}\label{eq:lem4-eq2}
\|\zeta_0\|_{L^2(\Omega)}^2
&=E_h(u,\mathcal{Q}_h\varphi)
+R_h(u;\xi,\mathcal{Q}_h\varphi)\nonumber\\
&\quad -s_h(\zeta_h, \mathcal{Q}_h\varphi)
+J_1+J_2+J_3.
\end{align}

By Cauchy-Schwarz inequality and (\ref{eq:sh}), we get
\begin{align}\label{eq:sh-est}
|s_h(\zeta_h, \mathcal{Q}_h\varphi)|
\leq s_h^{1/2}(\zeta_h, \zeta_h)s_h^{1/2}(\mathcal{Q}_h\varphi, \mathcal{Q}_h\varphi)
\leq Ch\3bar\zeta_h\3bar|\varphi|_{H^2(\Omega)}.
\end{align}

It follows from Cauchy-Schwarz inequality and (\ref{Pih-app}) that
\begin{align}\label{eq:J1-est}
|J_1|&\leq (\sum_{K\in\T_h}h_K^{-1}\|\zeta_0-\zeta_b\|_{L^2(\partial K)}^2)^{1/2}\nonumber\\
&\quad \times (\sum_{K\in\T_h}h_K\|\Pi_h(a(u)\nabla\varphi)-a(u)\nabla\varphi\|_{L^2(\partial K)}^2)^{1/2}\nonumber\\
&\leq Ch\3bar \zeta_h\3bar |\varphi|_{H^2(\Omega)},
\end{align}
and
\begin{align}\label{eq:J2-est}
|J_2|&\leq M_a\|\nabla\varphi-\Pi_h(\nabla\varphi)\|_{L^2(\T_h)}\|\nabla_w\zeta_h\|_{L^2(\T_h)}\nonumber\\
&\leq Ch\3bar \zeta_h\3bar |\varphi|_{H^2(\Omega)}.
\end{align}

From H\"{o}lder's inequality and (\ref{Pih-app}), it follows that
\begin{align}\label{eq:J3-est}
|J_3|&\leq M_a\|\nabla\varphi-\Pi_h(\nabla\varphi)\|_{L^2(\T_h)}\|\nabla u\|_{L^{\infty}(\Omega)}\|\zeta_0\|_{L^2(\Omega)}\nonumber\\
&\leq Ch\|\zeta_0\|_{L^2(\Omega)}|\varphi|_{H^2(\Omega)}.
\end{align}

Finally, appealing to Lemma \ref{lem:l6} and \ref{lem:l7}, and the estimates of (\ref{eq:sh-est}), (\ref{eq:J1-est}), (\ref{eq:J2-est}) and (\ref{eq:J3-est}), we have
\begin{align*}
\|\zeta_0\|_{L^2(\T_h)}^2
&\leq C[h^{k+1}|u|_{H^{k+1}(\Omega)}+h(\3bar\zeta_h\3bar+\|\zeta_0\|_{L^2(\T_h)}) \\
&\qquad + h^{2k}|u|_{H^{k+1}(\Omega)}^2]\|\varphi\|_{H^2(\Omega)},
\end{align*}
which together with (\ref{reg}) leads to
\begin{align*}
\|\zeta_0\|_{L^2(\T_h)}
&\leq C[h^{k+1}|u|_{H^{k+1}(\Omega)}+h(\3bar\zeta_h\3bar+\|\zeta_0\|_{L^2(\T_h)}) + h^{2k}|u|_{H^{k+1}(\Omega)}^2],
\end{align*}
which implies (\ref{eq:zeta-l2}) for sufficiently small $h$. The proof is completed.
\end{proof}

\begin{theorem}\label{thm:thm1}
Assume $u\in H^{k+1}(\Omega)\cap W_{\infty}^1(\Omega)$. For any $\xi\in \mathbb{B}_{h}(\mathcal{Q}_hu)$, let $\psi=\mathcal{F}_h(\xi)$ and denote by $\zeta_h=\mathcal{Q}_hu-\psi$. Then there exists a constant $C>0$, independent of $h$, such that
\begin{equation}\label{eq:thm1-1}
\3bar\zeta_h\3bar
\leq Ch^{k}|u|_{H^{k+1}(\Omega)},
\end{equation}
and
\begin{equation}\label{eq:thm1-2}
\|\zeta_0\|_{L^2(\T_h)}
\leq C(h^{k+1}+h^{2k}|u|_{H^{k+1}(\Omega)})|u|_{H^{k+1}(\Omega)},
\end{equation}
for sufficiently small $h$.
\end{theorem}
\begin{proof}
Plugging (\ref{eq:zeta-l2}) of Lemma \ref{lem:l5} back into (\ref{eq:zeta-3bar}) of Lemma \ref{lem:l4} yields
\begin{equation*}
\3bar\zeta_h\3bar
\leq C[h\3bar\zeta_h\3bar+h^{k}|u|_{H^{k+1}(\Omega)}
+h^{2k}|u|_{H^{k+1}(\Omega)}^2],
\end{equation*}
which implies
\begin{equation*}
\3bar\zeta_h\3bar
\leq Ch^{k}|u|_{H^{k+1}(\Omega)}(1+h^{k}|u|_{H^{k+1}(\Omega)})
\leq Ch^{k}|u|_{H^{k+1}(\Omega)},
\end{equation*}
for sufficiently small $h$.

Combining (\ref{eq:thm1-1}) with (\ref{eq:zeta-l2}) of Lemma \ref{lem:l5} leads to the $L^2$ norm estimate (\ref{eq:thm1-2}).
The proof is completed.
\end{proof}

According to (\ref{eq:thm1-1}) of Theorem \ref{thm:thm1}, we can easily conclude the following statements.
\begin{theorem}\label{thm:thm2}
For sufficiently small $h$, the map $\mathcal{F}_h$ maps $\mathbb{B}_{h}(\mathcal{Q}_hu)$ into itself.
\end{theorem}
\begin{proof}
For any $\xi\in \mathbb{B}_h(\mathcal{Q}_hu)$, from (\ref{eq:thm1-1}) of Theorem \ref{thm:thm1} it follows that
\begin{equation*}
\3bar\mathcal{Q}_hu-\mathcal{F}_h(\xi)\3bar
\leq Ch^{k}|u|_{H^{k+1}(\Omega)},
\end{equation*}
which implies $\mathcal{F}_h(\xi)\in \mathbb{B}_h(\mathcal{Q}_hu)$. Therefore, $\mathcal{F}_h(\mathbb{B}_h(\mathcal{Q}_hu))\subseteq \mathbb{B}_h(\mathcal{Q}_hu)$. The proof is completed.
\end{proof}

The following theorem shows that the mapping $\mathcal{F}_h$ is a contraction in the norm $\3bar\cdot\3bar$ of the ball $\mathbb{B}_h(\mathcal{Q}_hu)$.
\begin{theorem}\label{thm:thm3}
For any given $\xi, \theta\in\mathbb{B}_h(\mathcal{Q}_hu)$, there holds
\begin{align}
\3bar \mathcal{F}_h(\xi)-\mathcal{F}_h(\theta)\3bar
\leq \rho(h)\3bar \xi-\theta\3bar,
\end{align}
with $\rho(h)\in (0,1)$ for sufficiently small $h$.
\end{theorem}
\begin{proof}
For the sake of simplicity, let $\psi_h=\mathcal{F}_h(\xi)$ and $\phi_h=\mathcal{F}_h(\theta)$. From the definition (\ref{eq:f-map}) of $\mathcal{F}_h$, it is easy to see that
\begin{align}\label{eq:thm3-eq2}
D_h(u;\phi_h-\psi_h, v_h) = R_h(u;\xi, v_h)-R_h(u;\theta, v_h),\quad \forall\; v_h\in V_h^0.
\end{align}
From (\ref{eq:H}), we find
\begin{align*}
R_h(u;\xi, v_h)=(T_1(\xi), \nabla_w v_h)_{\T_h} + (T_2(\xi)\nabla u, \nabla_w v_h)_{\T_h},
\end{align*}
with
\begin{align*}
T_1(\xi) &= (a(\xi_0)-a(u))(\nabla_w\xi-\nabla u),\\[2pt]
T_2(\xi) &= a(\xi_0)-a(u)-(\xi_0-u)a_u(u).
\end{align*}
Similarly, we have
\begin{align*}
R_h(u;\theta, v_h)=(T_1(\theta), \nabla_w v_h)_{\T_h} + (T_2(\theta)\nabla u, \nabla_w v_h)_{\T_h}.
\end{align*}
Thus,
\begin{align}\label{eq:thm3-eq1}
&R_h(u;\xi, v_h)-R_h(u;\theta, v_h)\nonumber\\
&=(T_1(\xi)-T_1(\theta), \nabla_w v_h)_{\T_h} + ((T_2(\xi)-T_2(\theta))\nabla u, \nabla_w v_h)_{\T_h}\nonumber\\
&=I_1+I_2.
\end{align}
By the Taylor series expansions (\ref{eq:ts-1st}) and (\ref{eq:ts-2nd}), we get
\begin{align*}
&T_1(\xi)-T_1(\theta)\\
&=(a(\xi_0)-a(u))(\nabla_w\xi-\nabla u)-(a(\theta_0)-a(u))(\nabla_w\theta-\nabla u)\\
&=(a(\xi_0)-a(\theta_0))(\nabla_w\theta-\nabla u)
+(a(\xi_0)-a(u))\nabla_w(\xi-\theta)\\
&=(\xi_0-\theta_0)\widetilde{a}_u(\xi_0,\theta_0)(\nabla_w\theta-\nabla u)
+(\xi_0-u)\widetilde{a}_u(\xi_0,u)\nabla_w(\xi-\theta),
\end{align*}
and
\begin{align*}
T_2(\xi)-T_2(\theta)
&=a(\xi_0)-a(\theta_0)-(\xi_0-\theta_0)a_u(u)\\
&=[a(\xi_0)-a(\theta_0)-(\xi_0-\theta_0)a_u(\theta_0)]\\
&\qquad + (a_u(\theta_0)-a_u(u))(\xi_0-\theta_0)\\
&=(\xi_0-\theta_0)^2\widetilde{a}_{uu}(\xi_0,\theta_0)
+(\xi_0-\theta_0)(\theta_0-u)\widetilde{a}_{uu}(\theta_0,u).
\end{align*}
Then, by H\"{o}lder's inequality and the inverse inequality (\ref{inv}), we have
\begin{align}\label{eq:I1-est1}
|I_1|&\leq |((\xi_0-\theta_0)\widetilde{a}_u(\xi_0,\theta_0)(\nabla_w\theta-\nabla u), \nabla_w v_h)_{\T_h}| \nonumber\\
&\quad +|((\xi_0-u)\widetilde{a}_u(\xi_0,u)\nabla_w(\xi-\theta), \nabla_w v_h)_{\T_h}|\nonumber\\
&\leq M_a(\|\xi_0-\theta_0\|_{L^4(\T_h)}\|\nabla_w\theta-\nabla u\|_{L^2(\T_h)}\nonumber\\
&\quad +\|\xi_0-u\|_{L^4(\T_h)}\|\nabla_w(\xi-\theta)\|_{L^2(\T_h)})
\|\nabla_wv_h\|_{L^4(\T_h)}\nonumber\\
&\leq Ch^{-1/2}(\|\xi_0-\theta_0\|_{L^4(\T_h)}\|\nabla_w\theta-\nabla u\|_{L^2(\T_h)}\nonumber\\
&\quad +\|\xi_0-u\|_{L^4(\T_h)}\3bar\xi-\theta\3bar)
\|\nabla_wv_h\|_{L^2(\T_h)}.
\end{align}
Using Lemma \ref{lem:P-wg} and \ref{lem:happy}, one has
\begin{align*}
\|\xi_0-\theta_0\|_{L^4(\T_h)}
\leq C_P\|\xi-\theta\|_{1,h}
\leq C\3bar\xi-\theta\3bar,
\end{align*}
and
\begin{align*}
\|\xi_0-u\|_{L^4(\T_h)}&\leq C_P\|\xi-u\|_{1,h}\\
&\leq C_P(\|\xi-\mathcal{Q}_hu\|_{1,h}+\|\mathcal{Q}_hu-u\|_{1,h})\\
&\leq C(\3bar\xi-\mathcal{Q}_hu\3bar + h^k|u|_{H^{k+1}(\Omega)})\\
&\leq  Ch^k|u|_{H^{k+1}(\Omega)},
\end{align*}
where we have used $\xi\in\mathbb{B}_h(\mathcal{Q}_hu)$ in the last inequality.

By Lemma \ref{lem:com}, the triangle inequality and $\theta\in\mathbb{B}_h(\mathcal{Q}_hu)$, we obtain
\begin{align*}
\|\nabla_w\theta-\nabla u\|_{L^2(\T_h)}
\leq \3bar \theta-\mathcal{Q}_hu\3bar + \|\Pi_h(\nabla u)-\nabla u\|_{L^2(\T_h)}
\leq  Ch^k|u|_{H^{k+1}(\Omega)}.
\end{align*}
Plugging those estimates back into (\ref{eq:I1-est1}) yields
\begin{align}\label{eq:I1-est2}
|I_1|\leq Ch^{k-1/2}|u|_{H^{k+1}(\Omega)}\3bar\xi-\theta\3bar\3bar v_h\3bar.
\end{align}

Similarly, we have the following estimate for $I_2$:
\begin{align*}
|I_2|&\leq |((\xi_0-\theta_0)^2\widetilde{a}_{uu}(\xi_0,\theta_0)\nabla u,\nabla_w v_h)_{\T_h}|\nonumber\\
&\quad +|((\xi_0-\theta_0)(\theta_0-u)\widetilde{a}_{uu}(\theta_0,u)\nabla u,\nabla_w v_h)_{\T_h}|\nonumber\\
&\leq M_a|u|_{W_{\infty}^1(\Omega)}\|\xi_0-\theta_0\|_{L^4(\T_h)}(\|\xi_0-\theta_0\|_{L^4(\T_h)} +\|\theta_0-u\|_{L^4(\T_h)})\|\nabla_wv_h\|_{L^2(\T_h)}
\nonumber\\
&\leq Ch^{k}|u|_{H^{k+1}(\Omega)}\3bar\xi-\theta\3bar\3bar v_h\3bar,
\end{align*}
which together with (\ref{eq:I1-est2}) leads to
\begin{align}
|R_h(u;\xi, v_h)-R_h(u;\theta, v_h)|\leq Ch^{k-1/2}|u|_{H^{k+1}(\Omega)}\3bar\xi-\theta\3bar\3bar v_h\3bar.
\end{align}
Denote by $\chi_h=\phi_h-\psi_h$. Taking $v_h=\chi_h$ in (\ref{eq:thm3-eq2}) and then using the G{\aa}rding's inequality (\ref{eq:garding}), we find
\begin{align*}
\gamma (\3bar\chi_h\3bar^2+\|\chi_0\|_{L^2(\T_h)}^2)
&\leq \beta \|\chi_0\|_{L^2(\T_h)}^2 + D_h(u;\chi_h, \chi_h)\\
&\leq \beta \|\chi_0\|_{L^2(\T_h)}^2 + Ch^{k-1/2}|u|_{H^{k+1}(\Omega)}\3bar\xi-\theta\3bar\3bar \chi_h\3bar,
\end{align*}
which implies
\begin{align}\label{eq:chi-3bar}
\gamma \3bar\chi_h\3bar
&\leq \beta \|\chi_0\|_{L^2(\T_h)} + Ch^{k-1/2}|u|_{H^{k+1}(\Omega)}\3bar\xi-\theta\3bar.
\end{align}
Then an application of duality argument as in Lemma \ref{lem:l5} yields
\begin{align*}
\|\chi_0\|_{L^2(\T_h)}\leq C(h\3bar\chi_h\3bar+h^{k}|u|_{H^{k+1}(\Omega)}\3bar\xi-\theta\3bar),
\end{align*}
which combining with (\ref{eq:chi-3bar}) implies
\begin{align*}
\3bar\chi_h\3bar
\leq Ch^{k-1/2}|u|_{H^{k+1}(\Omega)}\3bar\xi-\theta\3bar.
\end{align*}
Let $\rho(h)=Ch^{k-1/2}|u|_{H^{k+1}(\Omega)}$. Since $k\geq 1$, then we have $\rho(h)<1$ for sufficiently small $h$. The proof is completed.
\end{proof}

\subsection{Error estimates}

According to Theorem \ref{thm:thm2} and \ref{thm:thm3} and Brouwer's fixed point theorem, we can deduce that $\mathcal{F}_h$ has a unique fixed point $u_h$ in $\mathbb{B}_h(\mathcal{Q}_hu)$, which is also the solution of the nonlinear WG finite element scheme (\ref{wg}). By the definition of the ball  $\mathbb{B}_h(\mathcal{Q}_hu)$, the finite element approximation $u_h$ of $u$ satisfies the bound
\begin{align}\label{eq:last1}
\3bar \mathcal{Q}_hu - u_h\3bar \leq Ch^k|u|_{H^{k+1}(\Omega)},
\end{align}
for sufficiently small $h$.
Furthermore,  Theorem \ref{thm:thm1} gives
\begin{equation}\label{eq:last2}
\|\mathcal{Q}_0u-u_0\|_{L^2(\T_h)}
\leq C(h^{k+1}+h^{2k}|u|_{H^{k+1}(\Omega)})|u|_{H^{k+1}(\Omega)},
\end{equation}
for sufficiently small $h$.

By the triangle inequality, Lemma \ref{lem:happy}, (\ref{Q0-app}), (\ref{eq:last1}) and (\ref{eq:last2}), we immediately have the following results.
\begin{theorem}\label{thm:main}
Suppose that $u\in H^{k+1}(\Omega)\cap W_{\infty}^1(\Omega)$.  Then there exists a constant $C>0$, independent of $h$, such that
\begin{equation*}
\|u-u_h\|_{1,h}\leq Ch^{k}|u|_{H^{k+1}(\Omega)},
\end{equation*}
and
\begin{equation*}
\|u-u_0\|_{L^2(\T_h)}
\leq C(h^{k+1}+h^{2k}|u|_{H^{k+1}(\Omega)})|u|_{H^{k+1}(\Omega)},
\end{equation*}
for sufficiently small $h$.
\end{theorem}

\section{Two-grid algorithm of the WG method}
In this section, we propose a two-grid algorithm of the WG method for the quasilinear elliptic problem (\ref{pde})-(\ref{pde-bc}). Let $\T_{\tau}$ and $\T_h$ be two shape-regular partitions of the domain $\Omega$, respectively, with different mesh sizes {$\tau$} and $h$ ($\tau>h$). And the corresponding WG finite element spaces $V_{\tau}$ and $V_h$ will be called coarse and fine space, respectively. In the applications given below, we shall always assume that
\[
 {\tau} = O(h^{\lambda})\quad\mbox{for some}\; 0<\lambda<1.
\]
With the discrete variational form $A_h(\cdot;\cdot,\cdot)$ defined in (\ref{bilin-form}), let us introduce a two-grid algorithm of the WG method (\ref{wg}) as follows:
\begin{algorithm}[Two-grid WG method]\label{alg:two-grid}
\begin{itemize}
\item[Step 1.] On the coarse mesh $\T_{\tau}$, seek $u_{\tau}\in V_{\tau}$ such that
\[A_{\tau}(u_{\tau};u_{\tau}, v)= (f, v_0),\quad \forall\; v=\{v_0, v_b\}\in V_{\tau}^0.\]\\
\item[Step 2.] On the fine mesh $\T_h$, seek $u^h\in V_h$ such that
\[A_h(u_{\tau}; u^h, v) = (f, v_0),\quad \forall\; v=\{v_0, v_b\}\in V_h^0.\]
\end{itemize}
\end{algorithm}
In this algorithm, we firstly use the WG metod solve the quasilinear elliptic problem on a coarse space $V_{\tau}$, and obtain a rough approximation $u_{\tau}\in V_h$; and then use $u_{\tau}$ to linearize the nonlinear scheme on the fine space $V_h$, and solve the resulting linearized problem to get $u^h\in V_h$.

In order to prove the convergence of Algorithm \ref{alg:two-grid}, we introduce the following two lemmas.

\begin{lemma}\label{lem:bnd}
Let $u\in H^2(\Omega)\cap W_{\infty}^1(\Omega)$ and $u_h\in V_h$ be the solutions of problem (\ref{pde})-(\ref{pde-bc}) and the discrete problem (\ref{wg}), respectively. Then, there holds
\begin{align*}
\|\nabla_wu_h\|_{L^{\infty}(\T_h)}\leq C(u),
\end{align*}
where $C(u)$ is a generic positive constant, which is independent of $h$ but dependent on $|u|_{H^{2}(\Omega)}$ and $|u|_{W_{\infty}^1(\Omega)}$.
\end{lemma}
\begin{proof}
By the triangle inequality and Lemma \ref{lem:com}, we have
\begin{align}\label{eq:sec4-eq2}
\|\nabla_wu_h\|_{L^{\infty}(\T_h)} \leq \|\nabla_w(u_h-\mathcal{Q}_hu)\|_{L^{\infty}(\T_h)}
+\|\Pi_h(\nabla u)\|_{L^{\infty}(\T_h)}.
\end{align}
Using the inverse inequality (\ref{inv}) and (\ref{eq:last1}), we get
\begin{align}\label{eq:sec4-eq3}
\|\nabla_w(u_h-\mathcal{Q}_hu)\|_{L^{\infty}(\T_h)}
\leq Ch^{-1}\|\nabla_w(u_h-\mathcal{Q}_hu)\|_{L^2(\T_h)}
\leq C|u|_{H^2(\Omega)}.
\end{align}
In view of the approximation property (\ref{Pih-app}), we have
\begin{align*}
\|\Pi_h(\nabla u)\|_{L^{\infty}(\T_h)}
&\leq
\|\Pi_h(\nabla u)-\nabla u\|_{L^{\infty}(\T_h)}+\|\nabla u\|_{L^{\infty}(\Omega)}\nonumber\\
&\leq C|u|_{H^2(\Omega)}+|u|_{W_{\infty}^1(\Omega)},
\end{align*}
which combining with (\ref{eq:sec4-eq2}) and (\ref{eq:sec4-eq3}) completes the proof.
\end{proof}

\begin{lemma}\label{lem:coer}
For a given $\phi\in H^1(\T_h)$, there holds
\begin{align*}
A_h(\phi; v_h, v_h) \geq \min\{\alpha_0, 1\}\3bar v_h\3bar, \quad \forall\; v_h \in V_h^0.
\end{align*}
\end{lemma}
\begin{proof}
The proof is trivial, we omit it here.
\end{proof}

Now, we are ready to prove convergence estimates in the $H^1$-like norm $\|\cdot\|_{1,h}$ for the two-grid algorithm of the WG method under the assumption $u\in H^{k+1}(\Omega)\cap W_{\infty}^1(\Omega)$.

\begin{theorem}\label{thm:two-grid}
Assume that $u\in H^{k+1}(\Omega)\cap W_{\infty}^1(\Omega)$. Let $u_h\in V_h$ and $u^h\in V_h$ be the solutions obtained by Algorithm \ref{alg:nonlin} and \ref{alg:two-grid}, respectively. For $\tau\ll1$, we have
\begin{align}
\label{eq:sec4-eq6}
\|u-u^h\|_{1,h}&\leq C(u)(h^k+\tau^{k+1}),
\end{align}
where $C(u)$ is a positive constant, which is independent of $h$ but dependent on $|u|_{H^{k+1}(\Omega)}$ and $|u|_{W_{\infty}^1(\Omega)}$.
\end{theorem}
\begin{proof}
From (\ref{wg}) and the Step 2 of Algorithm \ref{alg:two-grid}, we know that
\begin{align*}
A_h(u_{\tau}; u^h, v) = (f, v_0) = A_h(u_h; u_h, v),\quad \forall\; v\in V_h^0,
\end{align*}
and thus using the definition of $A_h(\cdot;\cdot,\cdot)$, we have
\begin{align}\label{eq:sec4-eq1}
A_h(u_{\tau}; u^h-u_h, v)
&= A_h(u_{\tau}; u^h, v)-A_h(u_{\tau}; u_h, v)\nonumber\\
&= A_h(u_h; u_h, v)-A_h(u_{\tau}; u_h, v)\nonumber\\
&= ((a(u_{h,0})-a(u_{\tau,0}))\nabla_w u_h, \nabla_w v)_{\T_h}\nonumber\\
&= S_1 + S_2,
\end{align}
with
\begin{align*}
S_1 &= ((a(u_{h,0})-a(u))\nabla_w u_h, \nabla_w v)_{\T_h},\\
S_2 &= ((a(u)-a(u_{\tau,0}))\nabla_w u_h, \nabla_w v)_{\T_h}.
\end{align*}
By the Taylor expansion (\ref{eq:ts-1st}) and H\"{o}lder's inequality, we obtain
\begin{align*}
|S_1| &= |(\widetilde{a}_u(u_{h,0}, u)(u_{h,0}-u)\nabla_w u_h, \nabla_w v)_{\T_h}|\nonumber\\
&\leq M_a \|u_{h,0}-u\|_{L^2(\T_h)}\|\nabla_w u_h\|_{L^{\infty}(\T_h)}\|\nabla_wv\|_{L^2(\T_h)},
\end{align*}
then using Lemma \ref{lem:bnd} and Theorem \ref{thm:main} yields
\begin{align}\label{eq:sec4-eq4}
|S_1|\leq C(u)h^{k+1}\3bar v\3bar.
\end{align}
In a similar way, we obtain
\begin{align*}
|S_2|\leq C(u)\tau^{k+1}\3bar v\3bar,
\end{align*}
which together with (\ref{eq:sec4-eq4}) and (\ref{eq:sec4-eq1}) leads to
\begin{align*}
A_h(u_{\tau}; u^h-u_h, v)
\leq C(u)\tau^{k+1}\3bar v\3bar.
\end{align*}
Taking $v=u^h-u_h$ in the above equation and using Lemma \ref{lem:coer} and \ref{lem:happy} yields
\begin{align*}
\|u_h-u^h\|_{1,h} &\leq C(u)\tau^{k+1},
\end{align*}
which is combining with  Theorem \ref{thm:main} completes the proof.
\end{proof}
\begin{remark}
Theorem \ref{thm:two-grid} suggests that optimal order of convergence in the two-grid WG method with $k=1$ can be achieved by employing $\tau=O(h^{1/2})$.
\end{remark}

\section{Numerical Experiments} \label{Sec:numeric}
In this section, we conduct some numerical experiments to verify the theoretical predication on the WG method (\ref{wg}) and to demonstrate the efficiency of the two-grid WG method described in Algorithm \ref{alg:two-grid}.

The WG scheme (\ref{wg}) yields a nonlinear algebraic system, which will be solved by Newton's method. In our numerical experiments, the initial guess for Newton's iteration is taken to be the zero solution, i.e., $u_h^{(0)}=0$. The Newton iteration is continued until a tolerance of $\3bar u_h^{(k+1)}-u_h^{(k)}\3bar < 10^{-12}$ is reached, where $u_h^{(k)}$ and $u_h^{(k+1)}$ are two successive iterative solutions, respectively.

The following two examples are used for our numerical experiments.
\begin{example}\label{eg1}
Consider the model problem (\ref{pde})-(\ref{pde-bc}) with $a(u)=1+u$ in the domain $\Omega=(0,1)^2$.
The source data $f$ and Dirichlet data $g$ are chosen so that the exact solution is
\[u=\sin(\pi x)\sin(\pi y).\]
\end{example}

\begin{example}\label{eg2}
Consider the model problem (\ref{pde})-(\ref{pde-bc}) with $a(u)=1+\sin(u)/2$ in the domain $\Omega=(0,1)^2$.
The source data $f$ and Dirichlet data $g$ are chosen so that the exact solution is
\[u=\varphi(x)\varphi(y),\]
with {$\varphi(x)= x(1-x)e^{2x}$}.
\end{example}

\subsection{Accuracy test for the WG method}
In this subsection, we conduct some numerical tests for Example 1 and 2 to validate the theoretical results predicated in Theorem \ref{thm:main} for the WG method (\ref{wg}).
We shall consider rectangular grids and polygonal grids.
A least-squares fit is conducted for computing the numerical order of convergence.
\begin{figure}[!h]
\centering
\subfloat[The first level grid with $4\times 4$ elements]
{\includegraphics[width=0.5\textwidth]{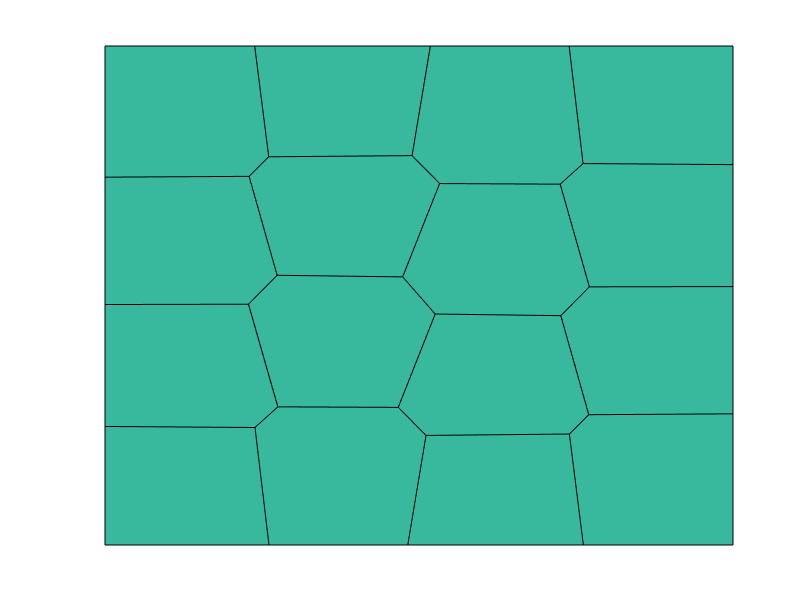}}
\subfloat[The second level grid with $8\times 8$ elements]
{\includegraphics[width=0.5\textwidth]{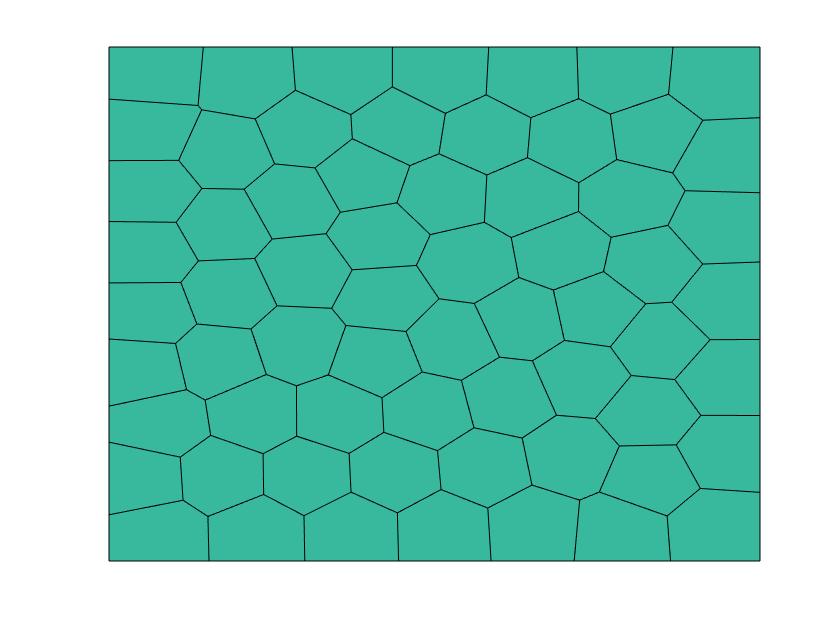}}
\caption{The demo of polygonal grids. }
\label{fig1}
\end{figure}

\begin{figure}[!h]
\centering
\subfloat[Rectangular grid with $4\times 4$ elements]
{\includegraphics[width=0.5\textwidth]{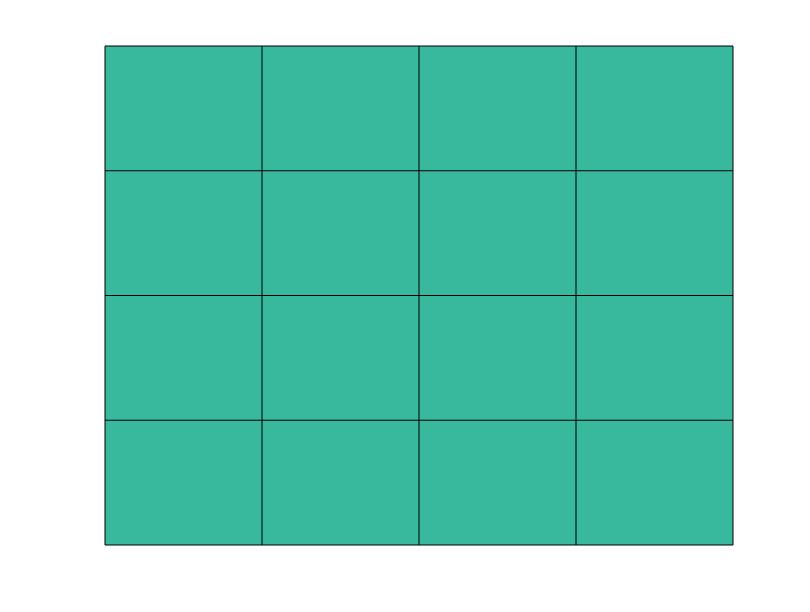}}
\subfloat[Rectangular grid with $8\times 8$ elements]
{\includegraphics[width=0.5\textwidth]{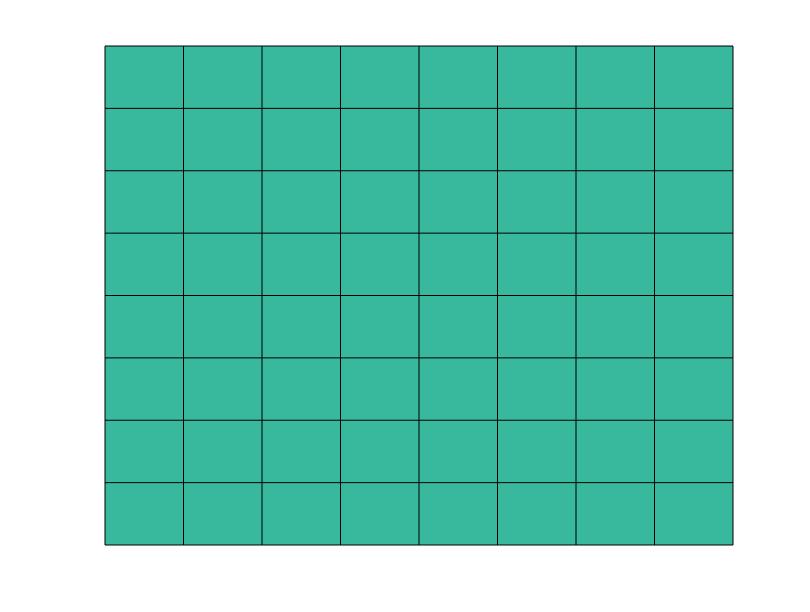}}
\caption{The demo of Rectangular grids. }
\label{fig2}
\end{figure}

Figure \ref{fig1} shows the first two level of polygonal grids used in our computation, which are generated by the MATLAB software package \texttt{PolyMesher} \cite{polymesh}.
The rectangular grid shown in Figure \ref{fig2} is obtained by dividing the domain $\Omega=(0,1)^2$ into $N\times N$ uniform small rectangles.

In Table \ref{t1}, we report the errors and the orders of convergence for the WG finite element solutions with $k=1$ and $k=2$ for Example 1. It is expected to see the order of convergence predicted by Theorem \ref{thm:main}, that is, $k$ for $\|u-u_h\|_{1,h}$ and $k+1$ for $\|u-u_0\|_{L^2}$, provided $k\geq 1$. We can see in Table \ref{t1} that these computed orders of convergence are full in agreement with the theoretical findings in Theorem \ref{thm:main}.

\begin{table}[h!]
  \centering \renewcommand{\arraystretch}{1.1}
  \caption{Error profiles and convergence rates of the WG solution for Example \ref{eg1}.}\label{t1}
\begin{tabular}{c|c|cc|cc}
\toprule[1.5pt]
\multirow{2}*{$k$}&\multirow{2}*{Mesh}&\multicolumn{2}{c|}{Polygonal grid}&\multicolumn{2}{c}{Rectangular grid}\\ \cline{3-6}
& & $\|u-u_h\|_{1,h}$&$\|u-u_0\|_{L^2}$&$\|u-u_h\|_{1,h}$&$\|u-u_0\|_{L^2}$     \\
\hline
 \multirow{5}*{1}
 &$4\times 4$&1.38E+00   &1.56E-01&1.63E+00 &2.05E-01\\
 &$8\times 8$&7.52E-01   &4.59E-02&8.66E-01 &5.78E-02\\
 &$16\times 16$&3.91E-01 &1.19E-02&4.39E-01 &1.48E-02\\
 &$32\times 32$&1.97E-01 &2.96E-03&2.20E-01 &3.74E-03\\
 &$64\times 64$&9.73E-02 &7.17E-04&1.10E-01 &9.35E-04\\
 \hline
 \multicolumn{2}{c|}{Rate}&0.96&1.95&0.97&1.95\\
 \hline
 \multirow{5}*{2}
 &$4\times 4$&4.59E-01 & 3.44E-02 &5.31E-01&4.37E-02\\
 &$8\times 8$&1.13E-01 &4.30E-03  &1.39E-01&5.44E-03\\
 &$16\times 16$&2.80E-02 &5.08E-04 &3.58E-02&6.65E-04\\
 &$32\times 32$&6.99E-03 &6.29E-05 &9.09E-03&8.21E-05\\
 &$64\times 64$&1.76E-03 &7.90E-06 &2.29E-03&1.02E-05\\
 \hline
 \multicolumn{2}{c|}{Rate}&2.01&3.03&1.97&3.02\\
 \bottomrule[1.5pt]
\end{tabular}%
\end{table}%

Table \ref{t2} shows the errors and the orders of convergence for the WG finite element solutions with $k=1$ and $k=2$ for Example 2. It is observed that these computed orders of convergence match with the theoretical order of convergence predicated in Theorem \ref{thm:main} again.

\begin{table}[h!]
  \centering \renewcommand{\arraystretch}{1.1}
  \caption{Error profiles and convergence rates of the WG solution for Example \ref{eg2}.}\label{t2}
\begin{tabular}{c|c|cc|cc}
\toprule[1.5pt]
\multirow{2}*{$k$}&\multirow{2}*{Mesh}&\multicolumn{2}{c|}{Polygonal grid}&\multicolumn{2}{c}{Rectangular grid}\\ \cline{3-6}
& & $\|u-u_h\|_{1,h}$&$\|u-u_0\|_{L^2}$&$\|u-u_h\|_{1,h}$&$\|u-u_0\|_{L^2}$     \\
\hline
 \multirow{5}*{1}
 &$4\times 4$&1.46E+00 &1.85E-01&1.58E+00&2.10E-01\\
 &$8\times 8$&6.84E-01 &3.98E-02&8.42E-01&5.57E-02\\
 &$16\times 16$&3.18E-01&8.38E-03&4.30E-01&1.42E-02\\
 &$32\times 32$&1.58E-01&2.05E-03&2.16E-01&3.56E-03\\
 &$64\times 64$&8.31E-02&5.55E-04&1.08E-01&8.92E-04\\
 \hline
 \multicolumn{2}{c|}{Rate}&1.03&2.10&0.97&1.97\\
 \hline
 \multirow{5}*{2}
 &$4\times 4$&3.26E-01 &2.16E-02 &3.59E-01 &2.64E-02\\
 &$8\times 8$&8.84E-02 &3.12E-03 &1.18E-01 &3.88E-03\\
 &$16\times 16$&2.40E-02 &4.29E-04 &3.34E-02 &5.07E-04\\
 &$32\times 32$&5.85E-03 &4.81E-05 &8.81E-03 &6.35E-05\\
 &$64\times 64$&1.57E-03 &6.44E-06 &2.25E-03 &7.92E-06\\
 \hline
 \multicolumn{2}{c|}{Rate}&1.93&2.94&1.84&2.93\\
 \bottomrule[1.5pt]
\end{tabular}%
\end{table}%

\subsection{Efficiency test for the two-grid WG method}
In this subsection, we shall carry out the two-grid WG method stated in Algorithm \ref{alg:two-grid} with $k=1$ on rectangular grids. We will take the coarse mesh size $\tau=h^{1/2}$, where $h$ is the fine mesh size. In this case, the theoretical result predicated by Theorem \ref{thm:two-grid} gives
\begin{align*}
\|u-u^h\|_{1,h}\leq Ch,
\end{align*}
where $u^h$ is the numerical solution obtained by employing the two-grid WG method on the fine mesh $\T_h$.

The computation is performed on a laptop with an Intel Core i7-7Y75 CPU at 1.60 GHz.
In Table \ref{t3} and \ref{t4}, we report the errors in $H^1$-like norm $\|\cdot\|_{1,h}$ of the WG method and the two-grid WG method on the same mesh level and their computational times (in second) for Example \ref{eg1} and \ref{eg2}.
It can be observed that (1) the two-grid WG method has the  convergence rate of $O(h)$ in the $H^1$-like norm $\|\cdot\|_{1,h}$, which is expected as Theorem \ref{thm:two-grid}; (2) both the WG method and the two-grid WG method have the similar accuracy on the same mesh level; (3) the cost of time for the two-grid WG method is significantly less than the one of WG method.

\begin{table}[h!]
  \centering \renewcommand{\arraystretch}{1.1}
  \caption{Comparisons of the errors and the CPU times of the WG method and the two-grid WG method for Example \ref{eg1}, $\tau=h^{1/2}$.}\label{t3}
\begin{tabular}{c|cc|cc}
\toprule[1.5pt]
\multirow{2}*{$h$}&\multicolumn{2}{c|}{WG method}&\multicolumn{2}{c}{Two-grid WG method}\\ \cline{2-5}
 & $\|u-u_h\|_{1,h}$&CPU time (sec.)&$\|u-u^h\|_{1,h}$&CPU time (sec.)  \\
\hline
 $1/4$&1.63E+00 &0.7031&1.66E+00&0.4375\\
 $1/16$&4.39E-01 &7.2031&4.76E-01&2.0469\\
 $1/36$&1.96E-01&33.4375&2.24E-01&6.8125\\
 $1/64$&1.10E-01&108.0156&1.28E-01&19.2500\\
 $1/100$&7.06E-02&271.5781&8.29E-02&53.8750\\
 \hline
 \multicolumn{1}{c|}{Rate}&0.98&&0.93&\\
 \bottomrule[1.5pt]
\end{tabular}%
\end{table}%

\begin{table}[h!]
  \centering \renewcommand{\arraystretch}{1.1}
  \caption{Comparisons of the errors and the CPU times of the WG method and the two-grid WG method for Example \ref{eg2}, $\tau=h^{1/2}$.}\label{t4}
\begin{tabular}{c|cc|cc}
\toprule[1.5pt]
\multirow{2}*{$h$}&\multicolumn{2}{c|}{WG method}&\multicolumn{2}{c}{Two-grid WG method}\\ \cline{2-5}
 & $\|u-u_h\|_{1,h}$&CPU time (sec.)&$\|u-u^h\|_{1,h}$&CPU time (sec.)  \\
\hline
 $1/4$&1.58E+00 &0.8125&1.57E+00&0.8594\\
 $1/16$&4.30E-01&7.8438&4.79E-01&2.9063\\
 $1/36$&1.92E-01&44.6563&2.25E-01&10.0000\\
 $1/64$&1.08E-01&111.9688&1.28E-01&23.9063\\
 $1/100$&6.93E-02&283.5781&8.26E-02&71.6719\\
 \hline
 \multicolumn{1}{c|}{Rate}&0.97&&0.91&\\
 \bottomrule[1.5pt]
\end{tabular}%
\end{table}%
{
\section{Conclusion}
In this paper,  a WG method is proposed and analyzed for
the quasi-linear elliptic problem of non-monotone type. With the use of {\color{red}Brouwer's} fixed point theorem, the existence of WG solution and error estimates in both the energy norm and the $L^2$ norm are derived. Moreover, a two-grid
WG method is introduced and its corresponding error estimates is analyzed
in the energy norm. Those theoretical results
are verified by numerical results. }

{\bf Declarations}\\[-5pt]

{\bf CRediT authorship contribution statement}\\

{\bf Peng Zhu}: Methodology, Software, Writing-review $\&$ editing. {\bf Shenglan Xie}: Conceptualization, Formal analysis, Writing-original draft

{\bf Conflict of interest} The authors declare no conflict interests.

\end{document}